\documentclass[]{amsart}

\usepackage{amsmath}
\usepackage{amssymb}
\usepackage{graphicx,color}
\usepackage{amsxtra}
\usepackage{mathtools}
\usepackage{enumerate}
\usepackage{nicefrac}
\mathtoolsset{showonlyrefs} 
\usepackage{color}
\usepackage{hyperref}
\hypersetup{
	colorlinks   = true, 
	urlcolor     = blue, 
	linkcolor    = blue, 
	citecolor   = red 
}

\newtheorem{theorem}{Theorem}[section]
\newtheorem*{theorem*}{Theorem}
\newtheorem{corollary}{Corollary}[theorem]
\newtheorem{Lemma}[theorem]{Lemma}
\newtheorem{proposition}[theorem]{Proposition}
\theoremstyle{definition}
\theoremstyle{definition}
\newtheorem{definition}[theorem]{Definition}

\theoremstyle{remark}
\newtheorem{remark}[theorem]{Remark}

\numberwithin{equation}{section}

\newcommand{\grads}{\nabla^{s}}
\newcommand{\mus}{\mu(N, s)}

\title[Besov regularity for the Bessel \( (p,s) \)-Laplacian]{Besov regularity of solutions to the Dirichlet problem for the Bessel \( (p,s) \)-Laplacian}

\author[J. P.~Borthagaray]{Juan Pablo Borthagaray}
\address[J. P.~Borthagaray]{PEDECIBA -- Program for the Development of Basic Sciences and Instituto de Matematica y Estad\'istica ``Rafael Laguardia'', Universidad de la Rep\'ublica, Montevideo, Uruguay.}
\email{jpborthagaray@fing.edu.uy}

\author[L. M. Del Pezzo]{Leandro M. Del Pezzo}
    \address[L. M. Del Pezzo]{PEDECIBA -- Program for the Development of Basic Sciences and  Instituto de Estad\'istica, Universidad de la Rep\'ublica, Montevideo, Uruguay.}
\email{leandro.delpezzo@fcea.edu.uy}

\author[J. C. Rueda Ni\~no]{Jos\'e Camilo Rueda Ni\~no}
\address[J. C. Rueda Ni\~no]{PEDECIBA -- Program for the Development of Basic Sciences and Instituto de Matematica y Estad\'istica ``Rafael Laguardia'', Universidad de la Rep\'ublica, Montevideo, Uruguay.}
\email{jcrueda@fing.edu.uy}

\begin{document}
\begin{abstract}
	We study the Dirichlet problem for a class of fractional $p$-Laplacian operators of order $s \in (0,1)$ defined through the Riesz fractional gradient, which differs fundamentally from the standard
    fractional $p$-Laplacian. Our analysis combines the framework of Lions-Calderón spaces, Besov embeddings, and an adaptation of Nirenberg’s difference quotient method, originally developed by Savaré \cite{savare1998regularity}, to the fractional Riesz setting. As a main result, we establish global Besov regularity estimates for weak solutions. Concretely, in the superquadratic regime \(p \ge 2\), we prove \(u \in \dot{B}_{p,\infty}^{\,s+1/p}(\Omega)\) for \(s \in [\tfrac{1}{p'},1)\), and \(u \in \dot{B}_{p,\infty}^{\,s+\frac{s}{p-1}}(\Omega)\) for \(s \in (0,\tfrac{1}{p'})\). In the subquadratic case \(1<p<2\), we show \(u \in \dot{B}_{p,\infty}^{\,s+1/2}(\Omega)\) for \(s \in [\tfrac{1}{2},1)\), and \(u \in \dot{B}_{p,\infty}^{\,2s}(\Omega)\) for \(s \in (0,\tfrac12)\), with quantitative bounds depending on the source data.
\end{abstract}

	\maketitle

	
	\section{Introduction}
The study of partial differential equations involving fractional-order operators has garnered significant attention in recent decades, driven by their capacity to model a wide range of nonlocal phenomena. Such operators naturally arise in diverse contexts including phase transitions \cite{wang2002metastability}, neural networks \cite{zhang2004existence}, population dynamics in biology \cite{carrillo2005spatial}, and elasticity theory \cite{silling2000reformulation,bellido2020fractional}, where classical local formulations fail to capture effects such as fracture propagation or detachment. Further significant applications include obstacle-type problems \cite{campos2023fractional,campos2024implicit}, fractional transport dynamics \cite{azevedo2024nonlocal}, machine learning and data analysis \cite{lu2024nonparametric}, and image processing \cite{gilboa2007nonlocal}, underscoring both the breadth and practical relevance of fractional-order models. Among the various non-local operators, fractional versions of the \( p \)-Laplacian have been a subject of intense research, extending the classical theory of quasi-linear elliptic equations to the non-local setting.

Fractional Sobolev spaces and the fractional Laplacian have been a subject of great interest in the study of nonlocal and fractional differential equations. However, for a long time, there was no suitable notion of a fractional gradient that could establish a direct connection with the fractional Laplacian and, at the same time, characterize Sobolev-type fractional spaces.

An important step forward was taken in 2015 by Shieh and Spector  \cite{shieh2015new}, who introduced the Riesz fractional gradient as a distributional operator. It is worth noting that this operator had been previously considered by Horváth in \cite{horvath1959some} with the aim of studying composition properties of the Riesz potential. Nonetheless, Horváth did not establish any connection between the operator he defined and a notion of fractional gradient.  Subsequently, \v{S}ilhav\'y~\cite{vsilhavyfractional} provided a structural characterization of the operator \(\grads \). More precisely, when restricted to the space of rapidly decaying Schwartz functions, \(\grads\) is uniquely determined up to a multiplicative constant by the properties of rotational invariance, translational invariance, and \(s\)-homogeneity.

It is worth emphasizing that homogeneity properties of different order s such as \(s\)-homogeneity for the fractional gradient and \(2s\)-homogeneity for the fractional Laplacian had already been investigated in earlier analytical frameworks. In particular, such scaling and invariance features naturally appear in the classical theory of singular integrals, as documented for instance in the work of Stein~\cite{stein1970singular}. This provides a historical analytical context in which operators of fractional order were studied well before the formulation of the Riesz fractional gradient in its modern variational form.
	
	In this paper, we focus on a particular class of fractional \( p \)-Laplacian equations defined through the Riesz fractional gradient, as introduced by Shieh and Spector in their foundational work \cite{shieh2015new}. Specifically, we investigate the regularity of weak solutions to the following Dirichlet problem in a bounded Lipschitz domain \( \Omega \subset \mathbb{R}^N \) ($N \ge 1$):
	\begin{equation}
		\label{Regularidad:Problema Dirichlet}
		\left\{
		\begin{aligned}
			-\Delta^{s}_{p} u &= f \quad \text{ in } \; \Omega \\
			u &= 0 \quad \text{ in } \; \Omega^c
		\end{aligned}
		\right. 
	\end{equation}
	where \( p \in (1, \infty) \), \( s \in (0,1) \), and the operator \( -\Delta^{s}_{p} u \) is defined as
	\[
	-\Delta^{s}_{p} u := - \operatorname{div}_{s}\left(|\nabla^{s} u|^{p-2} \nabla^{s} u\right).
	\]
    Here, $\operatorname{div}_{s}$ and $\nabla^s$ denote, respectively, the divergence and gradient of order $s$; their definitions, along with those of the related function spaces, are given in Section \ref{sec:preliminaries}. A weak solution \( u \in \widetilde{X}^{s,p}(\Omega) \) to \eqref{Regularidad:Problema Dirichlet} is understood as a minimizer of the energy functional \( \mathcal{J} : \widetilde{X}^{s,p}(\Omega) \rightarrow \mathbb{R} \), given by
	\begin{equation}
		\label{Regularidad: Operadores Formulación Debil}
		\mathcal{J}(v) = \frac{1}{p}\int_{\mathbb{R}^N} \left| \nabla^{s} v \right|^p \, dx - \langle f, v \rangle,
	\end{equation}
for a given source term \( f \) belonging to a suitable dual space, typically \( X^{-s,p'}(\Omega) := [\widetilde{X}^{s,p}(\Omega)]' \), where $\langle \cdot, \cdot \rangle$ denotes the corresponding duality pairing. The well–posedness of this problem, in particular the existence and uniqueness of weak solutions, was established in \cite{shieh2015new} under the assumption \( f \in L^{p'}(\Omega) \). However, by adapting the arguments developed therein, one can extend the result and prove an analogous well–posedness statement for 
general right–hand sides \( f \in X^{-s,p'}(\Omega) \).

In the variable–diffusivity case, the existence theory has recently been addressed by García-Sáez in \cite{garcia2025fractional}. In that work, the author introduces a new family of weighted fractional Sobolev spaces \( X^{s,p}_{0,w}(\Omega) \), defined for Muckenhoupt weights $w$ and which extend the unweighted Lions–Calder\'on framework to the fractional setting, and concludes by establishing existence results for a class of degenerate fractional elliptic equations with variable diffusivity.

	The regularity theory for solutions of fractional PDEs is a cornerstone for a deeper understanding of their qualitative properties and for the development of accurate numerical methods. For the specific operator \( -\Delta^{s}_{p} \) defined via the Riesz fractional gradient, Schikorra, Shieh, and Spector \cite{schikorra2018regularity} made significant progress by demonstrating \( C_{\text{loc}}^{s+\alpha}(\Omega) \) regularity for solutions to the homogeneous problem. Their key insight was to show that the function \( v := I_{1-s}u \), where \( I_{1-s} \) denotes the Riesz potential of order \( 1-s \), solves an inhomogeneous classical \( p \)-Laplacian equation, thereby allowing them to leverage the established interior regularity theory for classical local operators.
	
	The primary objective of the present work is to extend the regularity analysis to 
    global
    estimates for weak solutions to \eqref{Regularidad:Problema Dirichlet}.
	%
    To achieve this, our methodology hinges on a powerful adaptation of Nirenberg's difference quotient technique developed by Savaré \cite{savare1998regularity} for proving regularity for classical elliptic equations in Lipschitz domains.
    This technique has been successfully adapted to the fractional setting in several recent contributions. 
    Reference
    \cite{borthagaray2023besov} employed this approach to establish Besov regularity for the Dirichlet problem associated with the integral fractional Laplacian in Lipschitz domains. Subsequently, 
    \cite{borthagaray2024quasi} extended the use of Savaré-type estimates to a broader class of quasi-linear fractional-order operators, also in Lipschitz domains. Inspired by these advancements, we adapt and apply the difference quotient machinery to the specific structure of the Riesz fractional \( p \)-Laplacian operator considered in \eqref{Regularidad:Problema Dirichlet}.

In a forthcoming work \cite{BDPR25+EF}, we analyze and implement a finite element method to approximate problem \eqref{Regularidad:Problema Dirichlet}. As usual in finite element analysis, the derivation of convergence rates for our scheme relies on the regularity estimates obtained throughout this paper.


The remainder of this paper is organized as follows. Section \ref{sec:preliminaries} collects preliminary material. There, we introduce the Riesz fractional gradient, describe its main analytical properties, and define the associated variational framework based on Lions--Calderón spaces. We also discuss their relation with complex interpolation theory. Furthermore, we recall the definition of Besov spaces and collect the functional analytic tools that will be required throughout the paper. In addition, we introduce a localized translation operator and establish a key identity for the gradient of such translations, which will play a fundamental role in the subsequent analysis.

In Section~\ref{sec:functionals} we introduce the notion of functional regularity.
This framework allows us to analyze the regularity properties of the energy functional \(\mathcal{J}\), which is a crucial step in the proof of our main results. Section~\ref{sec:regularity} is devoted to the presentation of the regularity results obtained for weak solutions of problem~\eqref{Regularidad:Problema Dirichlet}. We obtain
two distinct regularity regimes, depending on the interplay between the differentiability parameter \(s\) and the integrability exponent \(p\).
We first state the regularity estimates corresponding to the higher differentiability regime (see Theorem~\ref{teo-reg-caso-opt}) and afterwards, in Theorem~\ref{teo-reg-subopt}, we describe the regularity behavior in the complementary regime. Succinctly, we prove 
\[ \begin{split}
&u \in \dot{B}_{p,\infty}^{s+\min\left\{\frac{s}{p-1},\frac1p\right\}}(\Omega),  \quad p \ge 2, \\
&u \in \dot{B}_{p,\infty}^{s+\min\left\{s,\frac12\right\}}(\Omega),  \quad 1< p < 2,
\end{split} \]
 with quantitative bounds depending on the source data.

Finally, Section~\ref{sec:complementary} presents two variants of these results. Concretely, we comment on the presence of a variable diffusivity and on the derivation of intermediate regularity estimates under weaker assumptions on the data.

	\section{Preliminaries} \label{sec:preliminaries}
	
	In this section, we introduce the notation, the functional setting, and several auxiliary results that will be fundamental in our analysis. We begin by defining the relevant function spaces and their norms, with a particular emphasis on Besov and Lions-Calderón spaces. We then present alternative characterizations of Besov spaces based on translation operators and difference quotients, which will be essential for quantifying regularity. Finally, we explore how these translation-based descriptions relate to the smoothness of functions and the behavior of solutions to the nonlocal problem under consideration.

	\subsection{Fractional calculus operators and Lions-Calder\'on spaces}
	 
	We now define the fractional gradient, the associated variational spaces, and the connection between these spaces and the Bessel potential spaces. This relationship will provide fundamental properties that shall be used throughout the article.
	
\begin{definition}
	Let \(s \in (0,1)\), \(f \in C_{c}^{\infty}\left(\mathbb{R}^{N}\right)\), and \(\Phi \in C_{c}^{\infty}\left(\mathbb{R}^{N}; \mathbb{R}^{N}\right)\). The \emph{Riesz fractional gradient} of $f$ of order $s$ is defined by
    \begin{equation} \label{eq:def-grads}
		\nabla^s f(x) \coloneq \mu(N, s) \int_{\mathbb{R}^{N}} \frac{f(x) - f(y)}{|y - x|^{N + s }}\, \frac{x - y}{|x-y|} \, dy,
	\end{equation}
	and, correspondingly, the \emph{Riesz fractional divergence} of $\Phi$ of order $s$ is defined by
\[
\operatorname{div}_{s} \Phi(x) \coloneq \mu(N, s) \int_{\mathbb{R}^{N}} \frac{\Phi(x) - \Phi(y)}{|x - y|^{N + s}} \cdot \frac{x - y}{|x-y|} \, dy.
	\]
   
    In both definitions, we have taken
    \[
	\mu(N, s) \coloneq \frac{2^{s} \Gamma\left(\frac{N + s + 1}{2}\right)}{\pi^{N / 2} \Gamma\left(\frac{1 - s}{2}\right)}.
	\]
\end{definition}

We observe that the constant \( \mu(N,s) \) satisfies the scaling inequality
\begin{equation} \label{eq:scaling-mu}
	0 < C_1(N) \le \frac{\mu(N,s)}{1 - s} \le C_2(N),
\end{equation}
for  \( C_1(N), \, C_2(N) > 0 \) depending only on the dimension.

The preceding definitions can be extended to locally integrable functions with a finite $W^{s,1}(\mathbb{R}^N)$ seminorms, cf. \cite[Lemma 2.1]{comi2019distributional}.

\begin{remark}
We note that, for $s=0$ and $j = 1 , \ldots, N$, definition \eqref{eq:def-grads} 
formally yields   
\[ \begin{split}
\left[\nabla^0 f(x)\right]_j & = \mu(N,0) \, \lim_{\varepsilon \to 0} \int_{\mathbb{R}^{N}\setminus B_\varepsilon (x)} \frac{f(y) - f(x)}{|y - x|^{N}}\frac{y_j - x_j}{|x-y|}\, dy \\
& = \mu(N,0) \, \lim_{\varepsilon \to 0} \int_{\mathbb{R}^{N}\setminus B_\varepsilon (0)} \frac{z_j \, f(x-z) }{|z|^{N + 1}}\, dz =: R_j f (x),
\end{split} 
\]
the $j$-th Riesz transform (cf. \cite[Chapter III]{stein1970singular}). In this sense, we expect that, for $s \in (0,1)$ the  fractional gradients of order $s$ interpolate between the vector-valued Riesz transform and the classical gradient (see Lemma \ref{Regularidad: Cota gradiente s de función regular}
and Remark \ref{rem:asymptotics-gradient} below).
\end{remark}

The following integration by parts formula is proven in \cite[Section 2.4]{comi2019distributional}.
We denote by $\mbox{Lip}_{c}$ the class of Lipschitz functions with bounded, compact support.

\begin{proposition}
	\label{Dualidad Gradiente y divergencia}
	Let \(s \in (0,1)\), 
     \(\varphi \in \mbox{Lip}_{c}\left(\mathbb{R}^{N}\right)\), and \(\Phi \in \mbox{Lip}_{c}\left(\mathbb{R}^{N}; \mathbb{R}^{N}\right)\).
    Then,
	\[
	\int_{\mathbb{R}^{N}} \varphi(x) \operatorname{div}_{s} \Phi(x)\, dx = - \int_{\mathbb{R}^{N}} \Phi(x) \cdot \nabla^{s} \varphi(x)\, dx.
	\]
\end{proposition}

Based on the fractional gradient, we introduce the associated variational space. Later, we will explore its relation with fractional Sobolev spaces.

\begin{definition}
	\label{Espacios Lions-Calderón}
	Let \( s \in (0,1) \) and \( p \in [1, \infty) \). The \emph{Lions–Calderón space} is defined as
	\[
	X^{s, p}(\mathbb{R}^{N}) \coloneq \left\{ f \in L^{p}(\mathbb{R}^{N}) : \nabla^{s} f \in L^{p}(\mathbb{R}^{N}; \mathbb{R}^{N}) \right\},
	\]
	endowed with the norm
	\[
	\|f\|_{X^{s, p}(\mathbb{R}^{N})} \coloneq \left( \|f\|_{L^{p}(\mathbb{R}^{N})}^{p} + \|\nabla^{s} f\|_{L^{p}(\mathbb{R}^{N}; \mathbb{R}^{N})}^{p} \right)^{1/p}.
	\]
\end{definition}

We now introduce the Lions–Calderón spaces on open domains.

\begin{definition} \label{def:Lions-Calderon-domain}
    Let \( \Omega \subset \mathbb{R}^{N}\) be a Lipschitz domain, \(s \in (0,1)\), and \(p \in [1, \infty)\).  
    The Lions–Calderón space on \(\Omega\) is defined as the quotient space of all functions in \( X^{s,p}(\mathbb{R}^{N}) \) that agree almost everywhere on \(\Omega\), endowed with the norm
    \[
        \|v \|_{X^{s,p} (\Omega)} \coloneq \inf \bigl\{ \|w\|_{X^{s,p}(\mathbb{R}^{N})} : w|_{\Omega} = v \bigr\}.
    \]
\end{definition}

In particular,  we consider the subspace of functions with vanishing exterior trace:
\begin{equation}
	\label{Regularidad: Espacios lions calderon tilde}
	\widetilde{X}^{s, p}(\Omega) \coloneq \left\{ w \in X^{s, p}(\mathbb{R}^N) : w = 0 \text{ in } \Omega^c := \mathbb{R}^N \setminus \Omega \right\}.
\end{equation}
Owing to a Poincaré-type inequality (see \cite[Theorem 2.9]{bellido2020gammaconvergencepolyconvexfunctionalsinvolving}), the space \( \widetilde{X}^{s,p}(\Omega) \) becomes a Banach space when endowed with
\[
\|f\|_{\widetilde{X}^{s, p}(\Omega)} \coloneq \|\nabla^{s} f\|_{L^{p}(\mathbb{R}^{N}; \mathbb{R}^{N})}.
\]

The dual of this space is denoted by
\[
X^{-s, p'}(\Omega) \coloneq \left( \widetilde{X}^{s, p}(\Omega) \right)'.
\]

As we shall comment below, the Lions-Calder\'on spaces coincide with Bessel potential spaces which, in turn, are defined through complex interpolation between Sobolev spaces of integer order of differentiability. To further motivate this result,
we now introduce a pointwise bound for the fractional gradient of Lipschitz continuous functions, which will be used in subsequent results, and that is in the spirit of an interpolation inequality.
	\begin{Lemma}
		\label{Regularidad: Cota gradiente s de función regular}
		Let \( \phi \in W^{1,\infty}(\mathbb{R}^N) \) and \( s \in (0,1) \). Then, the \( s \)-gradient of $\phi$ satisfies the pointwise bound
		\begin{equation*}
			|\nabla^s \phi(x)| \leq 
			C \left(\frac{1-s}{s}\right)^{1-s} \|\phi\|_{L^\infty(\mathbb{R}^N)}^{1-s}
			\|\nabla\phi\|_{L^\infty(\mathbb{R}^N)}^{s},
		\end{equation*}
		for all  \( x \in \mathbb{R}^N \), with a constant  that only depends on $N$.
	\end{Lemma}
	
	\begin{proof}
		We recall that the fractional gradient is defined by \eqref{eq:def-grads}, 
		\[
		\nabla^s \phi(x) = \mu(N,s) \int_{\mathbb{R}^N} \frac{(\phi(x) - \phi(y))(x - y)}{|x - y|^{N + s + 1}} \,dy.
		\]
		If $\phi$ is a constant function, then there is nothing to be proven. We thus assume $\| \nabla \phi \|_{L^\infty(\mathbb{R}^N)} > 0$.
		To estimate the fractional gradient of $\phi$, we let $R>0$ to be chosen and split the integral into two regions: a ball \( D_R(x) \) centered at \( x \) with radius $R$, and its complement \( D_R^c(x):= \mathbb{R}^N \setminus D_R(x) \).
		
		We estimate in  \( D_R(x) \) first. Using the Lipschitz continuity of \( \phi \) and integration in polar coordinates, we obtain:
		\begin{align*}
			\left| \int_{D_R(x)} \frac{(\phi(x) - \phi(y))(x - y)}{|x - y|^{N + s + 1}} \,dy \right|
			&\leq \|\nabla \phi\|_{L^\infty(\mathbb{R}^N)} \int_{D_R(x)} \frac{1}{|x - y|^{N + s - 1}} \,dy \\
			&\leq \frac{C_1(N) R^{1-s} \, \|\nabla \phi\|_{L^\infty(\mathbb{R}^N)}}{1-s}.
		\end{align*}
		
		Next, in \( D_R^c(x) \), we use the boundedness of \( \phi \) and again integrate in polar coordinates to obtain
		\begin{align*}
			\left| \int_{D_R^c(x)} \frac{(\phi(x) - \phi(y))(x - y)}{|x - y|^{N + s + 1}} \,dy \right|
			&\leq 2 \|\phi\|_{L^\infty(\mathbb{R}^N)} \int_{D_R^c(x)} \frac{1}{|x - y|^{N + s}} \,dy \\
			&\leq \frac{C_2(N)\|\phi\|_{L^\infty(\mathbb{R}^N)}}{s R^s}.
		\end{align*}
		
		Combining both contributions, choosing $R = \frac{C_2 (1-s) \|\phi\|_{L^\infty(\mathbb{R}^N)}}{C_1 s \|\nabla \phi\|_{L^\infty(\mathbb{R}^N)}}$, and property \eqref{eq:scaling-mu}, we conclude that:
		\[ \begin{split}
			|\nabla^s \phi(x)| & \leq \mu(N,s) \frac{C_1^s C_2^{1-s}}{(1-s)^s s^{1-s}} \, \|\phi\|_{L^\infty(\mathbb{R}^N)}^{1-s}
			\|\nabla\phi\|_{L^\infty(\mathbb{R}^N)}^{s} \\ &\le C(N) \left(\frac{1-s}{s}\right)^{1-s} \|\phi\|_{L^\infty(\mathbb{R}^N)}^{1-s}
			\|\nabla\phi\|_{L^\infty(\mathbb{R}^N)}^{s},
		\end{split} \]
		where $C$ is a constant that only depends on $N$.
	\end{proof}
	
\begin{remark} \label{rem:asymptotics-gradient}
	 The previous result remains asymptotically valid as \( s \to 0 \) or \( s \to 1 \). Indeed, for \( s = 1 \), we recover a trivial bound for the classical gradient while, for \( s = 0 \), the fractional gradients recover the vector-valued Riesz transform, which is not bounded in \( L^\infty(\mathbb{R}^N)\).
\end{remark}

Next, we make the relationship between fractional gradients and the complex interpolation more explicit. For this purpose, we introduce the Bessel potential spaces.


\begin{definition}
	Let \( p \in [1, \infty) \) and \( \alpha \in \mathbb{R} \). The Bessel potential space $\Lambda^{\alpha, p}(\mathbb{R}^n)$ is defined by
	\[
	\Lambda^{\alpha, p}(\mathbb{R}^n) \coloneq \left\{ u \in L^p(\mathbb{R}^n) : \mathcal{F}^{-1}\left[ (1+|\xi|^2)^{\frac{\alpha}{2}} \mathcal{F} u \right] \in L^p(\mathbb{R}^n) \right\},
	\]
	where \( \mathcal{F} \) denotes the Fourier transform and \( \mathcal{F}^{-1} \) its inverse. The norm on \( \Lambda^{\alpha, p}(\mathbb{R}^n) \) is given by
	\[
	\|u\|_{\Lambda^{\alpha, p}(\mathbb{R}^n)} \coloneq \left\| \mathcal{F}^{-1}\left[ (1+|\xi|^2)^{\frac{s}{2}} \mathcal{F} u \right] \right\|_{L^p(\mathbb{R}^n)}.
	\]
	For an open set \( \Omega \subset \mathbb{R}^n \),  we define
	\[
	\Lambda^{\alpha, p}(\Omega) \coloneq \left\{ u|_{\Omega} : u \in \Lambda^{\alpha, p}(\mathbb{R}^n) \right\},
	\]
	with the norm
	\[
	\|u\|_{\Lambda^{\alpha, p}(\Omega)} \coloneq \inf \left\{ \|v\|_{\Lambda^{\alpha, p}(\mathbb{R}^n)} : v|_{\Omega} = u \right\}.
	\]
\end{definition}

These spaces can be characterized via complex interpolation theory; we briefly recall this construction and refer to \cite[Theorem 3.7]{bellido2025bessel} for further details. The complex interpolation method relies on the theory of holomorphic vector-valued functions. Given a compatible couple of Banach spaces \( \overline{E} = (E_0, E_1) \), we consider the closed strip
\[
S \coloneq \{ z \in \mathbb{C} : 0 \leq \operatorname{Re} z \leq 1 \}.
\]
We define the space \( \mathfrak{F}(E_0, E_1) \) as the set of functions \( f: S \to E_0 + E_1 \) such that:

\begin{itemize}
	\item \( f \) is holomorphic in \( \operatorname{int}(S) \) and continuous and bounded on \( S \),
	\item for \( j = 0, 1 \), the function \( t \mapsto f(j + i t) \) is continuous from \( \mathbb{R} \) into \( E_j \), and satisfies
	\[
	\|f(j + i t)\|_{E_j} \to 0 \quad \text{as } |t| \to \infty.
	\]
\end{itemize}

Clearly, \( \mathfrak{F}(E_0, E_1) \) is a vector space. Moreover, equipped with the norm
\[
\|f\|_{\mathfrak{F}(\overline{E})} \coloneq \max \left\{ \sup_{t \in \mathbb{R}} \|f(i t)\|_{E_0}, \sup_{t \in \mathbb{R}} \|f(1 + i t)\|_{E_1} \right\},
\]
it becomes a Banach space. For each \( \theta \in [0, 1] \), the interpolation space \( [E_0, E_1]_\theta \) is then defined as the set of all \( x \in E_0 + E_1 \) such that there exists \( f \in \mathfrak{F}(\overline{E}) \) with \( f(\theta) = x \), and is endowed with the norm
\[
\|x\|_\theta \coloneq \inf \left\{ \|f\|_{\mathfrak{F}(\overline{E})} : f \in \mathfrak{F}(\overline{E}), \ f(\theta) = x \right\}.
\]

Let \( s \in (0,1) \) and \( p \in [1, \infty] \). Choosing \( E_0 = L^p(\Omega) \) and \( E_1 = W^{1,p}(\Omega) \), one obtains
\[
\Lambda^{s,p}(\Omega) \coloneq \left(L^{p}(\Omega), W^{1,p}(\Omega)\right)_{s}.
\]

In general, for \( s \in (0,1) \) and \( p \in (1, \infty) \), we have the identification
\[
X^{s,p}(\Omega) = \Lambda^{s,p}(\Omega)
\]
with equivalent norms, as shown, for example, in \cite[Proposition 2.4.3]{campos2021lions}.

\subsection{The Bessel $(p,s)$-Laplacian} \label{sec:(p,s)-Lap}
We now recall the quasilinear Dirichlet-type problem that will be the focus of our study: a fractional version of the $p$-Laplacian built from the Riesz fractional gradient. Let \( \Omega \subset \mathbb{R}^d \) be a bounded Lipschitz domain, and fix \( p \in (1, \infty) \), \( s \in (0,1) \). Given a source term \( f \in X^{-s, p'}(\Omega) \), we consider the nonlinear problem \eqref{Regularidad:Problema Dirichlet}. 
It
represents the strong-form Euler–Lagrange equation associated with the functional \eqref{Regularidad: Operadores Formulación Debil}. More precisely, one seeks minimizers of the functional $\mathcal{J} \colon \widetilde{X}^{s,p}(\Omega) \to \mathbb{R}$,

\begin{equation} \label{eq:functionals} 
\begin{split}
& \mathcal{J}(w) \coloneq \mathcal{J}_1(w) - \mathcal{J}_2(w),
\quad \text{where} \\
& \mathcal{J}_1(w) \coloneq \frac{1}{p} \int_{\mathbb{R}^d} \left|\nabla^s w\right|^p\,dx, 
\quad
\mathcal{J}_2(w) \coloneq \langle f, w \rangle.
\end{split}
\end{equation}
Thanks to the Poincaré-type inequality from~\cite[Theorem 2.9]{bellido2020gammaconvergencepolyconvexfunctionalsinvolving}, \( \mathcal{J}_1 \) defines a norm on $\widetilde{X}^{s,p}$ and is weakly lower semicontinuous, ensuring the existence of minimizers by the direct method in the calculus of variations.

When one computes the first variation at such a minimizer $u \in \widetilde{X}^{s, p}(\Omega)$, obtains the variational formulation
\begin{equation} \label{eq:weak-formulation}
\left\langle \left| \nabla^s u \right|^{p-2} \nabla^s u, \nabla^s v \right\rangle = \langle f, v \rangle \quad \text{for all } v \in \widetilde{X}^{s, p}(\Omega).    
\end{equation}

From this notion of weak solution, provided $u$ is regular enough, one can perform the integration by parts result provided by Proposition~\ref{Dualidad Gradiente y divergencia} to obtain the strong formulation from above.



\begin{remark}
	\label{rem: coercividad}
For the $p$-Laplacian-type operator under consideration, it is clear that coercivity holds when \( p \geq 2 \), while for \( 1<p<2 \) one only obtains $2$-coercivity on bounded sets. As in the case of the classical $p$-Laplacian,
these properties follow from an application of Hölder's inequality together with a few elementary inequalities in \( \mathbb{R}^N \). 
In addition, one can verify by similar arguments that the operator is continuous and monotone.  For a detailed proof of these facts,
we refer to \cite[Section~5.1]{glowinski1975approximation}; see also \cite[Example 3]{savare1998regularity}.
\end{remark}

\subsection{Relationship with the fractional Sobolev and Besov spaces}
In this section, we explore the connection between Lions–Calderón spaces and other intermediate spaces between integer-order Sobolev spaces. Concretely, we focus on spaces obtained by the real interpolation method: the fractional Sobolev (or Gagliardo-Slobodeckiĭ) or, 
more generally, Besov spaces.

	We now introduce the latter by real interpolation, as presented in \cite[Chapter~22]{tartar2007introduction}.
    Specifically, given a compatible pair of Banach spaces \( (X_0, X_1) \) and \( u \in X_0 + X_1 \), the Peetre \( K \)-functional is defined for \( t > 0 \) by
	\[
	K(t, u) \coloneq \inf \left\{  \|u_0\|_{X_0} + t\|u_1\|_{X_1}  : u = u_0 + u_1,\ u_0 \in X_0,\ u_1 \in X_1 \right\}.
	\]  
	
	For \( \theta \in (0,1) \) and \( q \in [1, \infty] \), the interpolation space is defined as  
	\[
	(X_0, X_1)_{\theta, q} \coloneq \left\{ u \in X_0 + X_1 : \|u\|_{(X_0, X_1)_{\theta, q}} < \infty \right\},
	\]  
	with the norm  
	\[
	\|u\|_{(X_0, X_1)_{\theta, q}} =
    	\left\{
	\begin{aligned}
		&\left(  \int_{0}^{\infty} (t^{-\theta} K(t,u))^{q}\, dt \right)^{1/q},  &\text{if } 1 \leq q < \infty, \\
		&\sup_{t > 0} \, t^{-\theta} K(t,u), &\text{ if  } q = \infty.
	\end{aligned}
    \right.
	\]  
	
	The Besov spaces in which we are interested in this work are defined by choosing \( X_0 = L^p(\Omega) \) and \( X_1 = W^{2,p}(\Omega) \), and considering \( \sigma \in (0,2) \), and \( p, q \in [1, \infty] \):
	\[
	B_{p, q}^{\sigma}(\Omega)  \coloneq \left(L^{p}(\Omega), W^{2,p}(\Omega)\right)_{\sigma / 2, q}, \quad
	\dot{B}_{p, q}^{\sigma}(\Omega)\coloneq \left\{ v \in B_{p, q}^{\sigma}(\mathbb{R}^{N}) : \operatorname{supp} v \subset \overline{\Omega} \right\}.
	\]
	For \( \sigma \in (0,1) \), an equivalent definition is  
	\[
	\dot{B}_{p, q}^{\sigma}(\Omega) = \left(L^{p}(\Omega), 
    W^{1,p}_0
    (\Omega)\right)_{\sigma, q}.
	\]  
    	Additionally, we define negative-order Besov spaces as 
	\[
	B_{p, q}^{-\sigma}(\Omega) \coloneq \left(L^{p}(\Omega), W^{-1,p}(\Omega)\right)_{\sigma, q}.
	\]  
	
	Two basic embeddings for Besov spaces on Lipschitz domains will be used later (cf. \cite[sections 3.2.4 and 3.3.1]             {triebel2010theory}): 
    \begin{align}
        B_{p, q_0}^{\sigma}(\Omega) \subset B_{p, q_1}^{\sigma}(\Omega), & \quad \text{if } \sigma > 0,\ 1 \leq p \leq \infty,\ 1 \leq q_0 \leq q_1 \leq \infty, \\ \label{eq: embedding besov}
		B_{p, q_1}^{\sigma_1}(\Omega) \subset B_{p, q_0}^{\sigma_0}(\Omega), & \quad \text{if } 0 < \sigma_0 < \sigma_1,\ 1 \leq p \leq q_0 \leq q_1 \leq \infty.
    \end{align}
	
	The following results are derived from the relationship between real and complex interpolation spaces; see \cite[Theorem 3.1]{peetre1969transformation}.
	\begin{Lemma}
		\label{Regularidad: Inmersión Lions Calderon Besov}
		Let \( \Omega \subset \mathbb{R}^d \) be a bounded Lipschitz domain, \( p \in [1, \infty) \), \( \sigma \in (0,1) \), and \( \varepsilon \in (0,1 - \sigma) \). Then,
		\[
		B_{p, \infty}^{\sigma + \varepsilon}(\Omega) \subset X^{\sigma,p}(\Omega),
		\]
		and there exists a constant \( C > 0 \), depending on \( \sigma \), \( \varepsilon \), and \( p \), such that
		\[
		\|v\|_{X^{\sigma,p}(\Omega)} \leq C \|v\|_{B_{p, \infty}^{\sigma + \varepsilon}(\Omega)} \quad \text{for all } v \in B_{p, \infty}^{\sigma + \varepsilon}(\Omega).
		\]
	\end{Lemma}
%
%
It is important to emphasize that the fractional Sobolev spaces and the Lions–Calderón spaces coincide only when \( p = 2 \), namely,
\[
	\widetilde{X}^{s,2}(\Omega) = \widetilde{W}^{s,2}(\Omega).
\]
This fact is a consequence of
Peetre’s interpolation theorem \cite[Theorem 3.1]{peetre1969transformation}, from which it follows that among Hilbert spaces all interpolation functors coincide. A detailed proof of this fact can be found in Bellido and García-Sáez \cite[Theorem 3.23]{bellido2025bessel}.

	Besov spaces can also be characterized via equivalent norms based on difference quotients over balls. For \( \lambda > 0 \), we define 
	\[
	\Omega_{\lambda} := \{x \in \Omega : \mathrm{dist}(x, \partial \Omega) > \lambda\}, \quad \Omega^{\lambda} := \{x \in \mathbb{R}^N : \mathrm{dist}(x, \partial \Omega) < \lambda\}.
	\]
	
	Let \( D = D_\rho(0) \) be the ball of radius \( \rho \) centered at the origin; we employ the letter $D$ to avoid confusion with the notation for Besov spaces. Given a function \( v \in L^p(\Omega) \) and 
    \( h \in D \), we define the translation \( v_h(x) := v(x+h) \). Then, for \( p,q \in [1, \infty) \) and \( \sigma \in (0,2) \), the Besov seminorm can be equivalently expressed as
	
\begin{equation}
\label{eq: Regularidad-Norma-Besov}
    	|v|_{B_{p, q}^{\sigma}(\Omega; D)} := \left(q \sigma(2-\sigma) \int_{D} \frac{\|v_h - 2v + v_{-h}\|_{L^{p}(\Omega_{|h|})}^{q}}{|h|^{d + q\sigma}} \, dh\right)^{1/q},
\end{equation}
	while for \( q = \infty \),
	\begin{equation}
	\label{eq: Regularidad-Norma-Besov-Inf}
	|v|_{B_{p, \infty}^{\sigma}(\Omega; D)} := \sup_{h \in D} \frac{\| v_h - 2v + v_{-h}\|_{L^p(\Omega_{|h|})}}{|h|^\sigma}.\end{equation}
	
	It is well-known that
    the norm \( \|\cdot\|_{L^p(\Omega)} + |\cdot|_{B_{p,q}^\sigma(\Omega; D)} \) is equivalent to the Besov norm \( \|\cdot\|_{B_{p,q}^\sigma(\Omega)} \) defined via interpolation; see \cite[Theorem 7.47]{adams2003jjf}. Furthermore, we use the result from \cite[Proposition 2.2]{borthagaray2023besov} which shows that the ball \( D \) can be replaced by a convex cone $C$ in the definition of Besov seminorms for \( q = \infty \) and \( \sigma \in (0,2) \). Thus, we obtain:
	\begin{equation}
		\label{Regularidad: Equivalecia normas Besov}
		|v|_{B_{p, \infty}^\sigma(\Omega; \, C)} \simeq |v|_{B_{p, \infty}^\sigma(\Omega; \, D)}.
	\end{equation}
	
	
	Next, we state the following auxiliary lemma, whose proof follows directly from either complex or real interpolation between the cases \( \sigma = 0 \) and \( \sigma = 1 \), i.e., between \( L^p(\Omega) \) and \( W^{1,p}(\Omega) \). 
	
	\begin{Lemma} 
	\label{Regularidad: Lemma estimativa de error}
		Let \( p \in [1, \infty] \), \( \sigma \in [0,1] \), and \( h \in D \). There exists a constant \( C > 0 \) such that for any Lipschitz domain \( \Omega \subset \mathbb{R}^N \), we have
		\[
		\|v - v_h\|_{L^p(\Omega)} \leq C |h|^\sigma \|v\|_{X^{\sigma, p}(\Omega^{|h|})} \quad \forall v \in X^{\sigma, p}(\Omega^{|h|}).
		\]
        Similarly, for all $q \in [1,\infty]$, we also have
        \begin{equation} \label{eq:traslacion_Besov}
            \|v - v_h\|_{L^p(\Omega)} \leq C |h|^\sigma \|v\|_{B^{\sigma, p}_q(\Omega^{|h|})} \quad \forall v \in B^{\sigma, p}_q(\Omega^{|h|}).
        \end{equation}
	\end{Lemma}
	
    Following an argument like in
    \cite[Proposition 2.1]{borthagaray2023besov} and using Lemma \ref{Regularidad: Lemma estimativa de error}, we can bound  Besov seminorms by differences of Lions–Calderón seminorms.
	
	\begin{proposition}
		\label{Resularidad: Reiteración Besov}
		Let \( s \in (0,1) \), \( p \in [1, \infty] \), \( \sigma \in [0,1] \), and \( \Omega \subset \mathbb{R}^N \) be a Lipschitz domain. Then,
		\[
		|v|_{B_{p,q}^{s+\sigma}(\Omega)} \lesssim 
		\begin{cases}
			\left(q(s+\sigma)(2-s-\sigma) \displaystyle\int_{D} \frac{\|v - v_h\|_{X^{s,p}(\Omega)}^q}{|h|^{N+q\sigma}} \, dh\right)^{1/q}, & \text{if } q \in [1, \infty), \\[1em]
			\displaystyle \sup_{h \in D} \frac{\|v - v_h\|_{X^{s,p}(\Omega)}}{|h|^\sigma}, & \text{if } q = \infty.
		\end{cases}
		\]
	\end{proposition}
	
	\begin{proof}
		Let \( t = s + \sigma < 2 \). From the properties \eqref{eq: Regularidad-Norma-Besov}-\eqref{eq: Regularidad-Norma-Besov-Inf}, we have
		\[
		|v|_{B_{p,q}^t(\Omega)} \simeq 
		\begin{cases}
			\left(q t(2 - t) \displaystyle \int_{D} \frac{\|v_h - 2v + v_{-h}\|_{L^p(\Omega_{|h|})}^q}{|h|^{N + q t}} \, dh\right)^{1/q}, & \text{if } q \in [1, \infty), \\[1em]
			\displaystyle \sup_{h \in D} \frac{\|v_h - 2v + v_{-h}\|_{L^p(\Omega_{|h|})}}{|h|^t}, & \text{if } q = \infty.
		\end{cases}
		\]
		Setting \( w = v_h - v \), we note that
        $v_h - 2v + v_{-h} = w - w_{-h}$ and,
		applying Lemma \ref{Regularidad: Lemma estimativa de error}, we obtain
		\[
		\|v_h - 2v + v_{-h}\|_{L^p(\Omega_{|h|})} \lesssim |h|^s \|v - v_h\|_{X^{s,p}(\Omega)}.
		\]
		The result follows by substituting into the Besov seminorm expression.
	\end{proof}
	
	The following result, from \cite[Lemma 2.6]{borthagaray2023besov}, shows that Besov seminorms can be equivalently expressed as sums of local norms over overlapping coverings.
	\begin{Lemma} \label{lemma:loc-Besov}
		Let \( p, q \in [1, \infty] \) and \( \sigma \in (0,2) \). Let \( \{ D_j \}_{j=1}^{J} \) be a finite covering of \( \Omega \) by balls of radius \( \rho \), i.e., \( D_j = D_\rho(x_j) \), and $v\in L^p(\Omega)$. Then, 
        \( v \in B^{\sigma}_{p,q}(\Omega) \) if and only if \( v|_{\Omega \cap D_j} \in B^{\sigma}_{p,q}(\Omega \cap D_j) \) for every \( j = 1, \ldots, J \), and
		\begin{equation}
			\|v\|^p_{B^{\sigma}_{p,q}(\Omega)} \simeq \sum_{j=1}^{J} \|v\|^p_{B^{\sigma}_{p,q}(\Omega \cap D_j)}.
		\end{equation}
		
		Moreover, let \( \delta \geq \rho \), and consider a finite covering \( \{ D_j \}_{j=1}^{J} \) of \( \Omega^\delta \). If \( v : \mathbb{R}^N \to \mathbb{R} \) satisfies \( \operatorname{supp}(v) \subset \overline{\Omega} \), then \( v \in \dot{B}^{\sigma}_{p,q}(\Omega) \) if and only if \( v|_{D_j} \in B^{\sigma}_{p,q}(D_j) \) for all \( j = 1, \ldots, J \), and
		\begin{equation}
			\|v\|^p_{\dot{B}^{\sigma}_{p,q}(\Omega)} \simeq |v|^p_{\dot{B}^{\sigma}_{p,q}(\Omega)} \simeq \sum_{j=1}^{J} |v|^p_{B^{\sigma}_{p,q}(D_j)}.
		\end{equation}
		
		The constants hidden in the equivalences depend on \( \sigma, p, q, \Omega \), and on the covering used.
	\end{Lemma}
	
	\subsection{Admissible directions and localized translations}
	
	We next present some classical yet useful results concerning Lipschitz domains.
	
	\begin{definition}
		For every \( x_0 \in \mathbb{R}^N \) and \( \rho \in (0,1] \), we define the set of admissible outward vectors by
		\begin{equation} \label{eq:def-calO}
		\mathcal{O}_\rho(x_0) \coloneqq \left\{ h \in D_\rho (0) :  (D_{2\rho}(x_0) \setminus \Omega) + t h \subset \Omega^c \text{ for all } t \in [0,1] \right\}.
\end{equation}
	\end{definition}
	
	A fundamental  aspect of bounded Lipschitz domains is that they satisfy the uniform cone property. This is stated in the following result; see \cite[Theorem 1.2.2.2]{grisvard1985elliptic}.
	
	\begin{proposition} \label{prop:uniform-cone}
		If \( \Omega \) is a bounded Lipschitz domain, then there exist \( \rho \in (0,1] \), \( \theta \in (0, \pi] \), and a measurable map \( \boldsymbol{n}: \mathbb{R}^N \to \mathbb{S}^{N-1} \) such that, for every \( x \in \mathbb{R}^N \),
		\[
		\mathcal{C}_\rho(\boldsymbol{n}(x), \theta) \coloneq \left\{ h \in \mathbb{R}^N : |h| \leq \rho,\ h \cdot \boldsymbol{n} \geq |h| \cos \theta \right\} \subset \mathcal{O}_\rho(x).
		\]
	\end{proposition}
	
	Global translations are not suitable to analyze the local behavior of solutions to our Dirichlet problem. Instead, following the approach of Savaré \cite{savare1998regularity}, we introduce localized translations in admissible directions. The localization is obtained via a convex combination of the identity with a translation, weighted by a cut-off function.
	
	\begin{definition}
		\label{Regularidad: Operador de Traslación Localizada}
		For \( v\colon \Omega \to \mathbb{R} \), let \(\tilde{v}\) denote its zero extension onto $\Omega^c$. For \( h \in \mathbb{R}^N \), define
		\[
		v_h(x) \coloneq \tilde{v}(x + h), \quad x \in \mathbb{R}^N.
		\]
		For \( x_0 \in \mathbb{R}^N \) and \( \rho > 0 \), let \( D_\rho(x_0) \) be the ball of radius \( \rho \) centered at \( x_0 \). With a smooth cut-off function \( \phi \) satisfying
		\[
		0 \leq \phi \leq 1, \quad \phi \equiv 1 \text{ in } D_\rho(x_0), \quad \text{supp}(\phi) \subset D_{2\rho}(x_0),
		\]
		we define the localized translation operator:
		\[
		T_h v \coloneq \phi v_h + (1 - \phi)v.
		\]
	\end{definition}
The introduction of the admissible set \eqref{eq:def-calO} guarantees that the operator \( T_h \) maps \( \widetilde{X}^{s,p}(\Omega) \) into itself for every \( h \in \mathcal{O}_\rho(x_0) \). We now turn to the task of characterizing the fractional gradient of a localized translation, expressing it as the localized translation of the gradient together with an additional commutator term.

\begin{Lemma}
	\label{Regularidad: Traslación del gradiente}
	Let \( v \in X^{s,p}(\mathbb{R}^{N}) \), and \( h \in \mathcal{O}_\rho\left(x_0\right) \). Then the fractional gradient of the localized translation \( T_h v \) satisfies the identity
	\begin{equation}
		\label{Regularidad: Identidad gradiente trasladado}
		\nabla^s(T_h v)(x) = T_h(\nabla^s v)(x) + \mathcal{C}_{\phi,v}(x),
	\end{equation}
	where the commutator term \( \mathcal{C}_{\phi,v} \) is given by 
	\begin{align}
		\label{eq: Regularidad-conmutador}
	   \mathcal{C}_{\phi,v}(x) \coloneq  \mus \int_{\mathbb{R}^N} \frac{(\phi(x) - \phi(y)) \, (v_h(y) - v(y)) \, (x - y)}{|x - y|^{N + s + 1}} \, dy.
	\end{align}
\end{Lemma}
\begin{proof}
	We begin by observing that the difference of the localized translations can be decomposed as 
	\[ 
	T_h v(x) - T_h v(y) = z(x,y)+ p(x,y),
	\]
	where
	\begin{align*}
		z(x,y) &\coloneq \phi(x)(v_h(x) - v_h(y)) + (1 - \phi(x))(v(x) - v(y)), \\
		p(x,y) &\coloneq (\phi(x) - \phi(y))\left(v_h(y)- v(y)\right).
	\end{align*}
	By the definition of the Riesz fractional gradient, we write
	\begin{align*}
		\nabla^s(T_h v)(x) 
		&= \mus \left( \int_{\mathbb{R}^N} \frac{z(x,y)(x - y)}{|x - y|^{N + s + 1}} \, dy + \int_{\mathbb{R}^N} \frac{p(x,y)(x - y)}{|x - y|^{N + s + 1}} \, dy \right).
	\end{align*}
	
From the linearity of the integral and regrouping terms, we obtain \eqref{Regularidad: Identidad gradiente trasladado}.
\end{proof}

We next aim to 
 obtain suitable $L^p$-bounds for the commutator $\mathcal{C}_{\phi,v}$, which will have a key role in our proof of regularity of solutions.
 
\begin{Lemma}
	\label{Regularidad: Cota conmutador}
 
 Let \( v \in \widetilde{X}^{s,p}(\Omega) \), \( \phi \in W^{1,\infty}(\mathbb{R}^N) \), \( s \in (0,1) \), and \( p \in [1, \infty) \). Then, the commutator \eqref{eq: Regularidad-conmutador} satisfies the estimate
	\begin{align}
		\label{Regularidad: Teorema de regularidad F_2 Cota conmutador con norma Lp}
\|\mathcal{C}_{\phi,v}\|_{L^p(\mathbb{R}^N)} &\leq C  \| v_h-v \|_{L^{p}(\mathbb{R}^N)} ,
	\end{align}
	where the constant \( C > 0 \) depends only on \( N, p, s \), and \( \|\phi\|_{W^{1,\infty}(\mathbb{R}^N)} \).
\end{Lemma}

\begin{proof} 
For notational convenience, we denote $\delta_h v := v_h - v$ throughout the proof. 
	From identity \eqref{eq: Regularidad-conmutador}, we can write
	\[
\|\mathcal{C}_{\phi,v}\|^p_{L^p(\mathbb{R}^N)}= \mus^p\int_{\mathbb{R}^N} \left|\int_{\mathbb{R}^N} \frac{(\phi(x)-\phi(y))\, \delta_h v(y) \, (x-y)}{|x-y|^{N+s+1}} dy\right|^p dx.
	\]	
	Let \( R > 0 \). We split the domain of integration into a near-field region \( D_R(x) \) and its complement:
	\begin{align*}
\|\mathcal{C}_{\phi,v}\|^p_{L^p(\mathbb{R}^N)} \leq C\bigg[&\int_{\mathbb{R}^N} \left|\int_{D_R(x)} \frac{(\phi(x)-\phi(y))\, \delta_h v(y) \,(x-y)}{|x-y|^{N+s+1}} dy\right|^p dx \\
		+ &\int_{\mathbb{R}^N} \left|\int_{D^c_R(x)} \frac{(\phi(y)-\phi(x))\, \delta_h v(y) \,(x-y)}{|x-y|^{N+s+1}} dy\right|^p dx\bigg].
	\end{align*}
	We estimate each term separately. For the near-field term over \( D_R(x) \), 
    since \( \phi \in W^{1,\infty}(\mathbb{R}^N) \), we 
    obtain 
	\begin{align}
\int_{\mathbb{R}^N} \Bigg|\int_{D_R(x)} 
&\frac{(\phi(x)-\phi(y))\, \delta_h v(y) \,(x-y)}{|x-y|^{N+s+1}}\,dy\Bigg|^p dx \\
&\leq \|\nabla \phi \|^{p}_{L^{\infty}(\mathbb{R}^N)} \int_{\mathbb{R}^N} \Bigg(\int_{D_R(x)} \frac{| \delta_h v(y)|}{|x-y|^{N+s-1}}\,dy\Bigg)^p dx \\
& = \|\nabla \phi \|^{p}_{L^{\infty}(\mathbb{R}^N)} \int_{\mathbb{R}^N} \Big( (|\delta_h v| \ast \mathcal{K}_R)(x)\Big)^p dx,\\
\end{align}
where 
\(
\mathcal{K}_R(z) := |z|^{-N-s+1}\,\chi_{D_R(0)}(z).
\)
In particular, a direct computation in polar coordinates shows that
\[
\|\mathcal{K}_R\|_{L^1(\mathbb{R}^N)}
= \frac{\omega_{N-1}}{1-s}\, R^{1-s},
\]
with \(\omega_{N-1}\) denoting the \((N-1)\)-dimensional surface measure of the unit sphere. 
With this bound at hand, we may apply Young's inequality for convolutions, to obtain
\begin{align}
    \label{eq: Conmutador-Dentro-Bola}
    \int_{\mathbb{R}^N} &\Bigg|\int_{D_R(x)} 
\frac{(\phi(x)-\phi(y))\, \delta_h v(y) \,(x-y)}{|x-y|^{N+s+1}}\,dy\Bigg|^p dx \\
&\leq \frac{\omega^p_{N-1}}{(1-s)^{p}}\, R^{(1-s)p} \|\phi \|^{p}_{W^{1,\infty}(\mathbb{R}^N)}  \| \delta_h v \|^p_{L^{p}(\mathbb{R}^N)}.
\end{align}
For the far-field contribution over \( D_R^c(x) \), we use the estimate 
\( |\phi(x)- \phi(y)|\leq 2 \| \phi\|_{L^{\infty}(\mathbb{R}^N)}\) and proceed in the same fashion, obtaining
\begin{align*}
		\int_{\mathbb{R}^N}   & \left|\int_{D_R^c(x)} \frac{(\phi(y)-\phi(x))\, \delta_h v(y) \,(x-y)}{|x-y|^{N+s+1}} dy \right|^{p} dx \\
        &\leq 2\|\phi \|^{p}_{L^{\infty}(\mathbb{R}^N)} \int_{\mathbb{R}^N} \Big( (|\delta_h v| \ast \mathcal{K}^c_R)(x)\Big)^p dx, 
\end{align*}
where 
the kernel is
\(
\mathcal{K}^c_R(z) := |z|^{-N-s}\,\chi_{D_R^c(0)}(z).
\)
In this case, a direct computation yields
\[
\|\mathcal{K}^c_R\|_{L^1(\mathbb{R}^N)}
= \frac{\omega_{N-1}}{s}\, R^{-s},
\]	
and therefore
\begin{align}
     \label{eq: Conmutador-fuera-Bola}
    \int_{\mathbb{R}^N} &\left|\int_{D^c_R(x)} \frac{(\phi(y)-\phi(x))\, \delta_h v(y) \,(x-y)}{|x-y|^{N+s+1}} dy\right|^p dx \\&\leq  2\|\phi \|^{p}_{L^{\infty}(\mathbb{R}^N)}\frac{\omega_{N-1}^p}{s^p}\, R^{-sp} \| \delta_h v \|^p_{L^{p}(\mathbb{R}^N)}.
\end{align}

Combining the bounds \eqref{eq: Conmutador-Dentro-Bola} and \eqref{eq: Conmutador-fuera-Bola}, we obtain \eqref{Regularidad: Teorema de regularidad F_2 Cota conmutador con norma Lp}.
\end{proof}

	Next, we state an estimate for the difference between a function and its localized translation. The proof follows by interpolation; see \cite[Lemma 2.9]{borthagaray2023besov}. 
	\begin{Lemma} 
    \label{lemm: diferencia-traslacion-v}
    Let \( T_h \) be as in Definition \ref{Regularidad: Operador de Traslación Localizada}. Then, for every \( h \in \mathbb{R}^N \), \( \sigma \in [0,1] \), \( q \in [1, \infty] \), and \( \gamma \in (-1,1) \), we have
		\[
		\|v - T_h v\|_{B^{\gamma}_{p,q}(D_{2\rho}(x_0))} \lesssim |h|^\sigma \|v\|_{B^{\gamma + \sigma}_{p,q}(D_{3\rho}(x_0))} \quad \text{for all } v \in B^{\gamma + \sigma}_{p,q}(D_{3\rho}(x_0)).
		\]
		
		The constant in the inequality is independent of \( \sigma \) and \( q \), but may blow up as \( \gamma \to \pm 1 \).
		
		Moreover, for every \( \sigma \in (0,1) \), we also have the estimate
		\[
		\|v - T_h v\|_{L^p(D_{2\rho}(x_0))} \lesssim |h|^\sigma \|v\|_{B^{\sigma}_{p,\infty}(D_{3\rho}(x_0))} \quad \text{for all } v \in B^{\sigma}_{p,\infty}(D_{3\rho}(x_0)).
		\]
	\end{Lemma}
	We now introduce a notion of regularity for functionals that quantifies their sensitivity with respect to a given family of perturbations. This concept, inspired by the framework proposed by Savaré in \cite{savare1998regularity}, plays a central role in our analysis.
	
	\begin{definition}
		\label{Regularidad: Definición Regularidad Local}
		Let \( V \) be a Banach space,  \( K \subset V \), and  \( \sigma > 0 \). Given a family of translation operators \( T_h: K \to K \), indexed by \( h \in D \subset \mathbb{R}^N \), we say that a functional \(  \mathcal{J}: V \to \mathbb{R} \) is \((T, D, \sigma)\)-\emph{regular} on \( K \) if, for every \( v \in K \), the following quantity is finite:
		\[
		\omega(v) := \omega(v; \mathcal{J}, T, D, \sigma) := \sup_{h \in D} \frac{\left| \mathcal{J}(T_h v) - \mathcal{J}(v) \right|}{|h|^{\sigma}} < \infty.
		\]
	\end{definition}
	
	Observe that the subadditivity of the functional 
    $\mathcal J$ from \eqref{eq:functionals} implies that it suffices to establish the \((T, D, \sigma)\)-regularity of individual components \( \mathcal{J}_1 \) and \( \mathcal{J}_2 \) to conclude the corresponding regularity of their sum. Based on the result of Savaré in \cite[Theorem 1 and Corollary 1]{savare1998regularity}, we obtain the following consequence of the monotonicity property of the functional (see Remark~\ref{rem: coercividad}).

	\begin{proposition}
		\label{Regularidad: Regularity and minimizers}
		Let \( x_0 \in \mathbb{R}^N \), \( \rho > 0 \), and \( h \in \mathcal{O}_\rho(x_0) \). Consider the translation operators \( T_h \) as defined in Definition~\ref{Regularidad: Operador de Traslación Localizada}. If \( u \) is a weak solution of the Dirichlet problem~\eqref{Regularidad:Problema Dirichlet} and \( \mathcal{J} \) is \((T, \mathcal{O}_\rho(x_0), \sigma)\)-regular on \( \widetilde{X}^{s,p}(\Omega) \), then:
		\begin{itemize}
			\item For \( p \geq 2 \), the following estimate holds:
			\begin{equation}
					\label{eq: regularity and minimizers p>2}
			 \| u - T_h u \|_{\widetilde{X}^{s,p}(\Omega)}^p \leq C \, \omega(u) \, |h|^{\sigma}.
			\end{equation}
			
			\item For \( 1 < p \leq 2 \), we have:
\begin{equation}
	\label{eq: regularity and minimizers p<2}
				 \| u - T_h u \|_{\widetilde{X}^{s,p}(\Omega)}^2 \leq C \left( \| u \|_{\widetilde{X}^{s,p}(\Omega)} + \| T_h u \|_{\widetilde{X}^{s,p}(\Omega)} \right)^{2-p} \omega(u) \, |h|^{\sigma}.
\end{equation}
		\end{itemize}
	\end{proposition}
	
	\begin{remark}
	    This result provides the theoretical foundation for our regularity method. The core of the strategy consists of the following steps.
	We begin by covering $\Omega$ with a finite set of balls of radius $\rho$, where $\rho$ is given by the uniform cone property that $\Omega$ satisfies (cf. Proposition \ref{prop:uniform-cone}).
    We select one such ball  \( D_\rho(x_0)\), centered at \( x_0\), and define the directional cone \( C_\rho = \mathcal{C}_\rho(\mathbf{n}(x_0), \theta) \). In the case \( p \in [2, \infty) \), if we take \( u \in \widetilde{X}^{s,p}(\Omega) \), then by combining Proposition~\ref{Resularidad: Reiteración Besov} and the identity~\eqref{eq: regularity and minimizers p>2} with the facts that \( T_h u = u_h \) on \( D_\rho(x_0) \), $u = T_h u = 0$ on $\Omega^c$ and Definition \ref{def:Lions-Calderon-domain},
     we obtain the following estimate:
	\begin{equation} \label{eq:reg-functional-pge2}
	\begin{aligned}
		|u|_{B_{p,\infty}^{s + \sigma/p}(D_\rho(x_0))}^p 
		& \lesssim \sup_{h \in D} \frac{\|u - u_h\|_{X^{s,p}(D_\rho(x_0))}^p }{ |h|^\sigma } \\
		& = \sup_{h \in D} \frac{ \| u - T_h u \|_{X^{s,p}(D_\rho(x_0))}^p}{ |h|^\sigma }\\
        &\lesssim \sup_{h \in D} \frac{ \| u - T_h u \|_{\widetilde{X}^{s,p}(\Omega)}^p}{ |h|^\sigma }\\
		& \lesssim \omega(u; \mathcal{J}, T, D, \sigma).
	\end{aligned}
	\end{equation}

	This chain of inequalities reveals the essential mechanism of the approach: the \((T, D, \sigma)\)-regularity of the functional \( \mathcal{J} \) implies local Besov regularity of the minimizers when \( p \geq 2 \). The case \( 1<p<2 \) can be handled in an analogous manner. Using 
	Proposition~\ref{Resularidad: Reiteración Besov} together with the identity 
	\eqref{eq: regularity and minimizers p<2}, we obtain
	\begin{equation} 
		\label{eq:reg-functional-p<2}
		\begin{aligned}
			|u|_{B_{p,\infty}^{s+\sigma/2}(D_\rho(x_0))}^{2} 
			&\lesssim\sup_{h \in D} \frac{|u - u_h|_{X^{s,p}(D_\rho(x_0))}^{2}}{|h|^\sigma} \\
			&\lesssim\sup_{h \in D} \frac{\|u - T_h u\|_{X^{s,p}(\Omega)}^{2}}{|h|^\sigma} \\
			&\lesssim \bigl(\|u\|_{\widetilde{X}^{s,p}(\Omega)} 
			+\|T_h u\|_{\widetilde{X}^{s,p}(\Omega)}\bigr)^{2-p}
			\,\omega(u; \mathcal{J}, T, D, \sigma).
		\end{aligned}
	\end{equation}
	This shows that, once we verify that \( \mathcal{J} \) satisfies  the \((T, D, \sigma)\)-regularity property of Definition~\ref{Regularidad: Definición Regularidad Local} for some \( \sigma \in (0,1] \), we obtain the corresponding  regularity of solutions.
	\end{remark}
	
	\section{Regularity of the functionals} \label{sec:functionals}
	
	In this section, we study separately the local regularity of the functionals \(\mathcal{J}_1\) and \(\mathcal{J}_2\) from \eqref{eq:functionals} in the sense of Definition~\ref{Regularidad: Definición Regularidad Local}. To this end, given an arbitrary point \(x_0 \in \mathbb{R}^N\), we denote by
	\[
	C_\rho = C_\rho(x_0), \quad D_\rho = D_\rho(x_0),
	\]
	the directional cone and the ball of radius \(\rho\) centered at \(x_0\), respectively. A key property we will exploit is that, for any \(v \in \widetilde{X}^{s,p}(\Omega)\) and any \(h \in C_\rho\), the translated function \(T_h v\) remains in \(\widetilde{X}^{s,p}(\Omega)\). This allows us to evaluate \(\mathcal{J}(T_h v)\) and compare it with \(\mathcal{J}(v)\).
	
	To complete the implementation of the regularity method, it remains to verify the \((T, D, \sigma)\)-regularity condition for each of the operators appearing in the functional~\eqref{Regularidad: Operadores Formulación Debil} with respect to the localized translation operators, see definition \eqref{Regularidad: Operador de Traslación Localizada}. Since $\text{supp}(\phi) \subset D_{2\rho}(x_0)$, the following identity and estimate for \(\mathcal{J}_1\) hold:
\begin{align}
\mathcal{J}_{1}\bigl(T_{h} v\bigr)-\mathcal{J}_{1}(v)
&= \int_{\mathbb{R}^N} f\,\phi\,(v_h-v)\,dx
= \int_{\mathbb{R}^N} v\,\bigl(f_{-h}\phi_{-h}-f\phi\bigr)\,dx \nonumber\\
&\le \int_{D_{3\rho}(x_0)} |v|\,\bigl|f_{-h}\phi_{-h}-f\phi\bigr|\,dx.
\end{align}

Combining the above estimate with Lemma~\ref{lemm: diferencia-traslacion-v}, and adapting the argument of \cite[Proposition~3.1]{borthagaray2024quasi} to our framework, we arrive at the following result for the linear term in the energy functional.

	\begin{proposition} \label{prop:reg-J1}
		Let \( q \in(1, \infty], \, p \in (1, \infty), \, \sigma \in(0,1],\) and \( f \in B_{p', q^{\prime}}^{\gamma}(\Omega) \) for some \( \gamma \in(-1,\sigma) \), then \( \mathcal{J}_{1} \) defined in \eqref{eq:functionals}  is \( \left(T, \mathcal{O}_{\rho}(x_0), \sigma\right) \)-regular in \( \dot{B}_{p, q}^{\sigma-\gamma}(\Omega) \) and
		\begin{equation}
			\label{Regularidad:Regularidad F_1}
			\sup _{h \in C_{\rho}} \frac{\mathcal{J}_{1}\left(T_{h} v\right)-\mathcal{J}_{1}(v)}{|h|^{\sigma}} \lesssim\|f\|_{B_{p^{\prime}, q^{\prime}}^{\gamma}(\Omega)}\|v\|_{B_{p, q}^{\sigma-\gamma}\left(D_{3 \rho}(x_0)\right)}, 
		\end{equation}
		for all \( v \in \dot{B}_{p, q}^{\sigma-\gamma}(\Omega). \)
	\end{proposition}
	    
	Next, we show that the operator \( \mathcal{J}_2 \) is \( \left(T, \mathcal{O}_\rho(x_0), \sigma\right) \)-regular.
\begin{proposition}
	\label{Regularidad: Regularidad F_2 Proposition}
	Let \( s \in (0,1) \) and \( \sigma \in (0,1] \). Then, the functional \( \mathcal{J}_2\), defined in \eqref{eq:functionals}, is \( (T, \mathcal{O}_\rho(x_0), \sigma) \)-regular on \( \widetilde{X}^{s,p}(\Omega) \) for every \( x_0 \in \Omega \) and \( \rho > 0 \). That is, for every \( v \in \widetilde{X}^{s,p}(\Omega) \), the following inequality holds:
	\begin{equation}
		\label{Regularidad: Regularidad F_2}
		\sup_{h \in \mathcal{O}_\rho(x_0)} \frac{\mathcal{J}_2(T_h v) - \mathcal{J}_2(v)}{|h|^{\sigma }} \lesssim \|v\|_{\dot{B}^\sigma_{p,\infty}(\Omega)} \|v\|_{\widetilde{X}^{s,p}(\Omega)}^{p-1} + \|\nabla^s v\|_{L^p(D_{3\rho}(x_0);\mathbb{R}^N)}^p.
	\end{equation}
\end{proposition}
\begin{proof}
  In first place, we have the following elementary inequality, with a constant $C = C(p)$, valid for all $u, v \in L^p(\mathbb{R}^N)$:
	\begin{equation}
		\label{Regularidad: Regularidad de F_2 cota auxiliar savare 1}
    \int_{\mathbb{R}^N} \big( |u(x) + v(x)|^p - |u(x)|^p \big) \, dx \leq C \|v\|_{L^p(\mathbb{R}^N)} \left( \|v\|_{L^p(\mathbb{R}^N)} + \|u\|_{L^p(\mathbb{R}^N)} \right)^{p-1}.
	\end{equation}
Additionally, noting that $T_h u = \phi u_h + (1-\phi) u$, using the convexity of \( F(x) = |x|^p \), performing a change of variables, and considering the supports of \( u \) and \( \phi \), we obtain
	\begin{equation}
		\label{Regularidad: Regularidad de F_2 cota auxiliar savare 2} \begin{split}
		\int_{\mathbb{R}^N} \left( |T_h u(x)|^p - |u(x)|^p \right)\, dx
        & \le \int_{D_{2\rho(x_0)}} \phi(x) \left(|u_h(x)|^p - |u(x)|^p \right) \, dx \\
        & \leq C |h| \|u\|_{L^p(D_{3\rho}(x_0))}^p.
\end{split}	\end{equation}

 Now, let \( v \in \widetilde{X}^{s,p}(\Omega) \). By Lemma~\ref{Regularidad: Traslación del gradiente}, 
 we can split $\mathcal{J}_2(T_h v) - \mathcal{J}_2(v)$ as
\[
\begin{aligned}
	  J_1 + J_2 :=& \int_{\mathbb{R}^N} \left( \left|T_h(\grads v) + \mathcal{C}_{\phi,v}\right|^p - \left|T_h(\grads v)\right|^p  \right)\, dx \\ &+ \int_{\mathbb{R}^N } \left(\left|T_h(\grads v)\right|^p - \left|\grads v\right|^p \right)\, dx   
\end{aligned}
\]

From inequality~\eqref{Regularidad: Regularidad de F_2 cota auxiliar savare 1}, we can estimate \( J_1 \) as
\[
J_1 \leq C\,\|\mathcal{C}_{\phi,v}\|_{L^p(\mathbb{R}^N)}\left( \|\mathcal{C}_{\phi,v}\|_{L^p(\mathbb{R}^N)} + \|T_h(\grads v)\|_{L^p(\mathbb{R}^N)}\right)^{p-1}.
\]

Using Lemma~\ref{Regularidad: Cota conmutador} together with \eqref{eq:traslacion_Besov},
we have
\begin{equation} \label{Regularidad: Teorema de regularidad F_2 Cota conmutador Besov}
\begin{aligned}
\|\mathcal{C}_{\phi,v}\|_{L^p(\mathbb{R}^N)} \leq C \| v_h-v \|_{L^{p}(\mathbb{R}^N)}  
	\leq C |h|^{\sigma} |v|_{B^{\sigma}_{p,\infty}(\Omega^{|h|})},
\end{aligned}    
\end{equation}
substituting estimates~\eqref{Regularidad: Teorema de regularidad F_2 Cota conmutador con norma Lp} and~\eqref{Regularidad: Teorema de regularidad F_2 Cota conmutador Besov} into the bound for \( J_1 \), we obtain the following chain of inequalities:
\[
\begin{aligned}
J_1 &\le C \,\|\mathcal{C}_{\phi,v}\|_{L^p(\mathbb{R}^N)}
\Bigl(\|\mathcal{C}_{\phi,v}\|_{L^p(\mathbb{R}^N)} + \|T_h(\nabla^s v)\|_{L^p(\mathbb{R}^N)}\Bigr)^{p-1}\\
&\le C\,|h|^{\sigma}\,|v|_{B^{\sigma}_{p,\infty}(\Omega^{|h|})}
\Bigl(\|v\|_{L^{p}(\Omega)} + \|T_h(\nabla^s v)\|_{L^p(\mathbb{R}^N)}\Bigr)^{p-1}.
\end{aligned}
\]

We use $ \|T_h(\nabla^s v)\|_{L^p(\mathbb{R}^N)} \le C  \|\nabla^s v\|_{L^p(\mathbb{R}^N)} = C \|v\|^{\,p-1}_{\widetilde{X}^{s,p}(\Omega)}$ and the Poincar\'e inequality (e.g. \cite[Theorem 2.9]{bellido2020gammaconvergencepolyconvexfunctionalsinvolving})
to obtain
\begin{align*}
    J_1 
        &\leq C\, |h|^{\sigma} \, |v|_{B^{\sigma}_{p,\infty}(\Omega^{|h|})} \, \|v\|^{\,p-1}_{\widetilde{X}^{s,p}(\Omega)}.
\end{align*}
\
Regarding the second term,
\[
J_2 = \int_{\mathbb{R}^N} \left( \left|T_h(\grads v)\right|^p - \left|\grads v\right|^p\right) \, dx,
\]
we apply inequality~\eqref{Regularidad: Regularidad de F_2 cota auxiliar savare 2} to deduce
\[
J_2 \leq C|h|\|\grads v\|^p_{L^p(D_{3\rho}(x_0);\mathbb{R}^N)}.
\]

Combining the bounds for \( J_1 \) and \( J_2 \), we deduce the following estimate for \( I_1 \):
\[
I_1 \leq C |h|^{\sigma} \left[ |v|_{B^{\sigma}_{p,\infty}(\Omega^{|h|})} \|v\|^{p-1}_{\widetilde{X}^{s,p}(\Omega)}  + |h|^{1-\sigma} \|\grads v\|^p_{L^p(D_{3\rho}(x_0);\mathbb{R}^N)} \right].
\]

Thus, grouping terms appropriately, we conclude the following estimate for the functional \( \mathcal{J}_2 \):
\[
\sup_{h \in \mathcal{O}_\rho\left(x_0\right)} \frac{\mathcal{J}_{2}\left(T_h v\right) - \mathcal{J}_{2}(v)}{|h|^{\sigma}} \lesssim  \|v\|_{\dot{B}^{\sigma}_{p,\infty}(\Omega)}\|v\|^{p-1}_{\widetilde{X}^{s,p}(\Omega)}+\|\grads v\|^{p}_{L^{p}(D_{3\rho}(x_0);\mathbb{R}^N)}.
\]
\end{proof}

	\section{Regularity of solutions}
    \label{sec:regularity}
	Having verified the regularity of each component functional, we are now prepared to derive regularity results for minimizers of the energy functional given in \eqref{Regularidad: Operadores Formulación Debil}. We obtain estimates in the Besov scale.
	
	\begin{theorem}
		\label{teo-reg-caso-opt}
		Let \(\Omega\) be a bounded Lipschitz domain, \( \sigma \in (0,1)\)  and \(u \in \widetilde{X}^{s,p}(\Omega)\) be a weak solution to \eqref{Regularidad:Problema Dirichlet}. If \(p \geq 2, s \in [\frac{1}{p'},1) \) and \(f \in B_{p^{\prime},1}^{-s+\frac{1}{p^{\prime}}}(\Omega)\) then \(u \in \dot{B}_{p, \infty}^{s+\frac{1}{p}}(\Omega)\) and
			\begin{equation}
				\|u\|_{\dot{B}_{p, \infty}^{s+\frac{1}{p}}(\Omega)} \lesssim\|f\|_{B_{p^{\prime}, 1}^{-s+\frac{1}{p^{\prime}}}(\Omega)}^{\frac1{p-1}}.
			\end{equation}

		If \(  1 < p <2, \,  s \in[\frac{1}{2},1)  \) and \(f \in B_{p^{\prime},1}^{-s+\frac{1}{2}}(\Omega)\), then  \(u \in \dot{B}_{p, \infty}^{s+\frac{1}{2}}(\Omega)\) and
			\begin{equation}
				\|u\|_{\dot{B}_{p, \infty}^{s+\frac12}(\Omega)} \lesssim \|f\|_{X^{-s,p^{\prime}}(\Omega)}^{\frac{2-p}{p-1}}\|f\|_{B_{p^{\prime}, 1}^{-s+\frac{1}{2}}(\Omega)}.
			\end{equation}
		
Above, all hidden constants depend on \( N,s,p, \) and \( \Omega. \)
	\end{theorem}
	\begin{proof}
		We begin by recalling that the Dirichlet problem \eqref{Regularidad:Problema Dirichlet} has a unique solution for any \(f \in X^{-s,p'}(\Omega)\). In particular, when the data satisfies our assumption \(f \in B_{p',1}^{-s+\tau}(\Omega)\), $\tau = \max \{ \frac{1}{p'}, \frac12 \}$,
        the continuous embedding \(B_{p',1}^{-s+\tau}(\Omega) \hookrightarrow X^{-s,p'}(\Omega)\) guarantees well-posedness, and we have the stability estimate
		\begin{equation}
			\label{Regularidad: Cota de Estabilidad}
			\|u\|_{\widetilde{X}^{s,p}(\Omega)} \lesssim \|f\|_{X^{-s,p'}(\Omega)}^{\frac{1}{p-1}} \lesssim \|f\|_{B_{p',1}^{-s+\tau}(\Omega)}^{\frac{1}{p-1}}.
		\end{equation}
		
		Since \(\Omega\) is a Lipschitz domain, 
        by Proposition \ref{prop:uniform-cone} there exist \(\rho > 0\) and \(\theta \in (0, \pi/2)\) such that
        $\mathcal{C}_\rho(\boldsymbol{n}(x), \theta) \subset \mathcal{O}_\rho(x),
		$ where $\boldsymbol{n} \colon \mathbb{R}^N \to S^{N-1}$ is a measurable map.
       Using this geometric property, we construct a finite covering of the expanded domain \(\Omega^\rho\) by balls \(D_\rho(x_j)\) of radius \(\rho\), centered at points \(\{x_j\}_{j=1,\ldots,J} \subset \mathbb{R}^N\). For each ball \(D_\rho(x_j)\), we consider the associated cone of admissible translation directions \(\mathcal{C}_j := \mathcal{C}_\rho(\mathbf{n}(x_j), \theta)\). 
		
		For exponents \(\sigma \in (0,1]\) and \(\gamma \in (-1,\sigma)\), we analyze the regularity properties of the functional \(\mathcal{J}\) through its decomposition into \(\mathcal{J}_1\) and \(\mathcal{J}_2\) as in \eqref{eq:functionals}. The subadditivity of the regularity modulus allows us to combine the estimates from \eqref{Regularidad:Regularidad F_1} for \(\mathcal{J}_1\) and \eqref{Regularidad: Regularidad F_2} for \(\mathcal{J}_2\), yielding the following bound:
	\begin{equation}\label{Regularidad: Cota de regularidad para F}
        \begin{aligned}
			\omega\left(u ; \mathcal{J}, T, \mathcal{C}_{j}, \sigma\right) 
			\lesssim &\|u\|_{\dot{B}^{\sigma}_{p,\infty}(\Omega)}\|u\|^{p-1}_{\widetilde{X}^{s,p}(\Omega)}+\|\grads u\|^{p}_{L^{p}(D_{3\rho}(x_0); \mathbb{R}^N)} \\ & +\|f\|_{B_{p^{\prime}, 1}^{\gamma}\left(\Omega\right)}\|u\|_{B_{p, \infty}^{\sigma-\gamma}\left(D_{3 \rho}\left(x_{j}\right)\right)}.
		\end{aligned}\end{equation}
Next, we distinguish between $p \ge 2$ or $1<p<2$.
        
		\textbf{Case \( p \in [2, \infty). \)}	
        Our starting point is inequality \eqref{eq:reg-functional-pge2}, which shows that the Besov seminorms of weak solutions can be controlled by the regularity modulus:
		\[
		|u|_{B_{p,\infty}^{s+\sigma/p}(D_\rho(x_j))}^p \lesssim \omega(u;\mathcal{J},T,\mathcal{C}_j,\sigma).
		\]
		This, together with \eqref{Regularidad: Cota de regularidad para F} and setting $\gamma = -s + \frac1{p'} \in (-s,0)$, yields
        \[ \begin{split}
		|u|_{B_{p, \infty}^{s+\sigma / p}\left(D_\rho\left(x_j \right)\right)}^p \lesssim &\|u\|_{\dot{B}^{\sigma}_{p,\infty}(\Omega)}\|u\|^{p-1}_{\widetilde{X}^{s,p}(\Omega)}+\|\grads u\|^{p}_{L^{p}(D_{3\rho}(x_0); \mathbb{R}^N)}  \\
        & +\|f\|_{B_{p^{\prime}, 1}^{-s + \frac1{p'}}\left(\Omega\right)}\|u\|_{B_{p, \infty}^{\sigma+s - \frac1{p'}}\left(D_{3 \rho}\left(x_{j}\right)\right)}.
		\end{split}\]
		
		Summing over \(j\) for \(j=1, \ldots, J,\) and using the localization property from Lemma \ref{lemma:loc-Besov}, we obtain 
        the following  inequality with a constant depending on the cardinality of the covering of $\Omega^\rho$:
		\begin{align}
			\label{Regularidad: Desigualdad general para el bootstraping}
			\|u\|_{\dot{B}_{p, \infty}^{s+\sigma / p}\left(\Omega\right)}^p \lesssim& \|u\|_{\dot{B}^{\sigma}_{p,\infty}(\Omega)}\|u\|^{p-1}_{\widetilde{X}^{s,p}(\Omega)}+\|\grads u\|^{p}_{L^{p}(\Omega)}\\&+\|f\|_{B_{p^{\prime}, 1}^{-s + \frac1{p'}}\left(\Omega\right)}\|u\|_{\dot{B}_{p, \infty}^{\sigma+s - \frac1{p'}}\left(\Omega\right)}
		\end{align}
		for all \(\sigma \in (0,1]\). To set up the bootstrap argument, it suffices to note that 
		\( s \geq \frac{1}{p'} \), hence the embedding  \(\dot{B}_{p,\sigma}^{\sigma+s - \frac1{p'}}(\Omega) \hookrightarrow \dot{B}_{p,\sigma}^{\sigma}(\Omega)\) holds, which yields
		\begin{align}
			\label{Regularidad: Desigualdad paso inductivo}
			\|u\|_{\dot{B}_{p, \infty}^{s+\sigma / p}\left(\Omega\right)}^p \lesssim & C_1 \|u\|_{\dot{B}^{\sigma+s-\frac{1}{p'}}_{p,\infty}(\Omega)}\|u\|^{p-1}_{\widetilde{X}^{s,p}(\Omega)}+C_2\|u\|^{p}_{\widetilde{X}^{s,p}(\Omega)} \\&+C_3\|f\|_{B_{p^{\prime}, 1}^{\frac{1}{p'}-s}\left(\Omega\right)}\|u\|_{\dot{B}_{p, \infty}^{\sigma+s-\frac{1}{p'}}\left(\Omega\right)},
		\end{align}		
		where the hidden constants depend on the covering (namely, on \(\Omega\) and \(J\)), and on \(N, s, p.\) Furthermore, observe that the differentiability index on the left-hand side is larger than that on the right-hand side. To exploit this fact, we define the sequence \(\{\sigma_k\}\) recursively as
	\begin{equation}
		\label{eq-rec-p>2-s>1p'}
			s+\frac{\sigma_{k}}{p}=\sigma_{i+1}+s-\frac{1}{p'} \Rightarrow \sigma_{k+1}:=\frac{\sigma_{k}}{p}+\frac{1}{p'},
	\end{equation}
		with initial value \(\sigma_{0}=0\); the latter is a consequence of setting the starting value \(\sigma_{0}+\gamma=s\) in \eqref{Regularidad: Desigualdad general para el bootstraping} with \(\gamma=s-\frac{1}{p'}\). Using an induction argument, we readily see that 
         \(\sigma_{k}=1- \left(\frac{1}{p}\right)^{k} \to 1\). 
        We claim that, for all $k\ge0$, \(u \in \dot{B}_{p, \infty}^{s+\sigma_{k} / p}(\Omega)\) and
		\begin{equation}
			\label{Regularidad: Afirmación por inducción}
			\|u\|_{\dot{B}_{p, \infty}^{s+\sigma_{k} / p}(\Omega)} \leq \Lambda_{k}\|f\|^{\frac{1}{p-1}}_{B_{p^{\prime}, 1}^{-s+\frac{1}{p'}}(\Omega)},
		\end{equation}
		for some uniformly bounded constants \(\Lambda_{k}\). We argue by induction. Note that the case \(k=0\) follows from the stability estimate \eqref{Regularidad: Cota de Estabilidad} and the embedding \(\widetilde{X}^{s,p}(\Omega) \subset \dot{B}_{p, \infty}^{s}(\Omega)\), giving
		\[
		\|u\|_{\dot{B}_{p, \infty}^{s}(\Omega)} \leq \Lambda_{0}\|f\|_{B_{p^{\prime}, 1}^{-s+\frac{1}{p^{\prime}}}(\Omega)}^{\frac{1}{p-1}}
		\]
		where \(\Lambda_{0}:=\Lambda_{0}(N, s, p, \Omega)\). We now set \(\sigma = \sigma_{k+1}\) in the inductive step inequality \eqref{Regularidad: Desigualdad paso inductivo}. Applying the inductive hypothesis yields the chain of inequalities:
		
		\[
		\begin{aligned}
			\|u\|_{\dot{B}_{p, \infty}^{s+\sigma_{k+1}/p}(\Omega)}^p 
			& \lesssim  C_1 \|u\|_{\dot{B}^{s+\frac{\sigma_k}{p}}_{p,\infty}(\Omega)}\|u\|^{p-1}_{\widetilde{X}^{s,p}(\Omega)}+C_2\|u\|^{p}_{\widetilde{X}^{s,p}(\Omega)} \\ & \quad + C_3\|f\|_{B_{p^{\prime}, 1}^{-s+\frac{1}{p'}}(\Omega)}\|u\|_{\dot{B}_{p, \infty}^{\frac{\sigma_k}{p}+s}(\Omega)} \\
			& \lesssim C_2 \|f\|^{p'}_{B_{p^{\prime}, 1}^{-s+\frac{1}{p'}}(\Omega)}+(C_1+C_3)  \Lambda_k \|f\|^{p'}_{B_{p^{\prime}, 1}^{-s+\frac{1}{p'}}(\Omega)}\\
			& \lesssim \left(C_2 + (C_1+C_3) \Lambda_k\right) \|f\|^{p'}_{B_{p^{\prime}, 1}^{-s+\frac{1}{p'}}(\Omega)}.
		\end{aligned}
		\]
		This establishes that the inductive claim \eqref{Regularidad: Afirmación por inducción} holds with \(\Lambda_{k+1} := (C_2 + (C_1+C_3) \Lambda_k)^{1/p}\). To prove uniform boundedness of this sequence, we define the majorant \(\Lambda := \max\{\Lambda_0, p'C_2+(C_1+C_3)^{p'} \}\). The base case \(\Lambda_0 \leq \Lambda\) holds by construction, and if \(\Lambda_k \leq \Lambda\), Young's inequality gives:	
		\[
		\begin{aligned}
			\Lambda_{k+1}^p &= C_2 + (C_1+C_3) \Lambda_k \leq  C_2 + \frac{(C_1+C_3)^{p'}}{p'}  +   \frac{\Lambda^{p}_{k}}{p} \\
			& \leq \frac{1}{p'}\left(p'C_2+(C_1+C_3)^{p'}\right) +\frac{\Lambda^{p}_k}{p} \leq \Lambda^p.
		\end{aligned}
		\]
		
	Taking the limit \(k \to \infty\), and noting that \(\sigma_k \to 1\), completes the proof for the case \(s \geq \frac{1}{p'}\).

		\textbf{Case \( p \in (1,2) \).} As a starting point for the subquadratic case, we rely on inequality \eqref{eq:reg-functional-p<2}, which yields the following local bound for Besov seminorms:
		
		\[
		\begin{aligned}
			|u|_{B_{p,\infty}^{s+\sigma/2}(D_\rho(x_0))}^{2} &\lesssim (\|u\|_{\widetilde{X}^{s,p}(\Omega)} + \|T_h u\|_{\widetilde{X}^{s,p}(\Omega)})^{2-p} \omega(u; \mathcal{J}, T, D, \sigma),
		\end{aligned}
		\]
	 invoking the stability estimate \eqref{Regularidad: Cota de Estabilidad}, we infer that
		\[
		|u|_{B_{p,\infty}^{s+\sigma/2}(D_\rho(x_0))}^{2} \lesssim \|f\|_{X^{-s,p'}(\Omega)}^{\frac{2-p}{p-1}} \omega(u; \mathcal{J}, T, D, \sigma).
		\]
		
		Next, using the inequality \eqref{Regularidad: Cota de regularidad para F}, setting \(\gamma=s-\frac{1}{2}\), and summing over \(j\) for \(j=1, \ldots, J\), and using the localization property of Besov norms, we obtain the inequality
		\begin{align*}
			|u|_{B_{p,\infty}^{s+\sigma/2}(\Omega)}^{2} 
			&\lesssim \|f\|_{X^{-s,p'}(\Omega)}^{\frac{2-p}{p-1}} \Bigg(
			C_1 \|u\|_{\dot{B}^{\sigma+s-\frac{1}{2}}_{p,\infty}(\Omega)} 
			\|u\|^{p-1}_{\widetilde{X}^{s,p}(\Omega)} \\
			&\hspace{4em} + C_2 \|u\|^{p}_{\widetilde{X}^{s,p}(\Omega)} 
			+ C_3 \|f\|_{B_{p^{\prime},1}^{-s+\frac{1}{2}}(\Omega)} 
			\|u\|_{\dot{B}_{p,\infty}^{\sigma+s-\frac{1}{2}}(\Omega)} 
			\Bigg),
		\end{align*}
     for every \(\sigma \in (0,1)\) and \(s \geq \tfrac{1}{2}\); note that this assumption is required to ensure the embedding \(\dot{B}_{p,\infty}^{\sigma+s - \frac{1}{2}}(\Omega) \hookrightarrow \dot{B}_{p,\infty}^{\sigma}(\Omega).\)
		
		Similarly to the previous case, we now proceed to exploit the improvement of the differentiability index from the right-hand side to the left one. We set \(\{\sigma_k\}\) recursively by
		\[
		s+\frac{\sigma_{k}}{2}=\sigma_{k+1}+s-\frac{1}{2} \quad \Rightarrow \quad \sigma_{k+1}=\frac{\sigma_{k}}{2}+\frac{1}{2},
		\]
		with initial value \(\sigma_{0}=0\). This yields \( \sigma_{k}=1-\frac{1}{2^{k}} \), which satisfies the conditions \(  \sigma_{k}+\gamma=\sigma_{k}+s-\frac{1}{2} \geq s>0 \) for all \( i \geq 1 \), and \( \sigma_{k} \rightarrow 1\). We prove by induction that
		
		\begin{equation}
			\label{Regularidad: Teorema Regularidad caso p entre 1 y 2}
			\|u\|_{\dot{B}_{p, \infty}^{s+\frac{\sigma_{k}}{2}}(\Omega)} \leq \Lambda_{k}\|f\|_{X^{-s,p^{\prime}}(\Omega)}^{\frac{2-p}{p-1}+2^{-k}}\|f\|_{B_{p^{\prime}, 1}^{-s+\frac{1}{2}}(\Omega)}^{1-2^{-k}} \quad \forall k \ge 0,
		\end{equation}
		for some uniformly bounded constants \(\Lambda_{k}\).  Note that the case \(k=0,\) follows by \eqref{Regularidad: Cota de Estabilidad} and the continuity of the embedding \( \widetilde{X}^{s,p} (\Omega) \subset \dot{B}^{s}_{p,\infty} (\Omega) \)	 we have:
		
		\begin{equation*}
			\|u\|_{\dot{B}_{p, \infty}^{s}(\Omega)} \leq \Lambda_{0}\|f\|_{{X}^{-s,p^{\prime}}(\Omega)}^{\frac{2-p}{p-1}}.
		\end{equation*}
		We now set \(\sigma = \sigma_{k+1}\), then we have
		\begin{align*}
			|u|_{B_{p,\infty}^{s+\frac{\sigma_{k+1}}{2}}(D_\rho(x_0))}^{2} &\lesssim \|f\|_{X^{-s,p'}(\Omega)}^{\frac{2-p}{p-1}} \Bigg( C_2\|f\|^{p'}_{\widetilde{X}^{-s,p'}(\Omega)}  \\ & 
             \hspace{10em} +(C_3+C_1)\|f\|_{B_{p^{\prime}, 1}^{-s+\frac{1}{2}}\left(\Omega\right)}\|u\|_{\dot{B}_{p, \infty}^{\sigma+s-\frac{1}{2}}\left(\Omega\right)}\Bigg)\\
			& \leq \|f\|_{X^{-s,p'}(\Omega)}^{\frac{2(2-p)}{p-1}+2^{-k}} \left(C_2\|f\|^{2-2^{-k}}_{\widetilde{X}^{-s,p'}(\Omega)}+(C_{1}+C_3) \Lambda_{k}\|f\|_{B_{p^{\prime}, 1}^{-s+\frac{1}{2}}(\Omega)}^{2-2^{-k}}\right)\\
			&\leq \|f\|_{X^{-s,p'}(\Omega)}^{\frac{2(2-p)}{p-1}+2^{-k}}\|f\|_{B_{p^{\prime}, 1}^{-s+\frac{1}{2}}(\Omega)}^{2-2^{-k}} \left(C_2 C_4^{2-2^{-k}}+(C_1+C_3)\Lambda_{k}\right),
		\end{align*}
		
	\noindent where \( C_{4} := C_{4}(\Omega, N, s, p, \beta) \) denotes the constant associated with the continuous embedding \( X^{-s,p'}(\Omega) \subset B_{p^{\prime}, 1}^{-s+\frac{1}{2}}(\Omega) \). Since \( 1 \leq 2 - 2^{-k} \leq 2 \), it follows that
		\[
		C_{4}^{2 - 2^{-k}} \leq \max\{C_{4}, C_{4}^2\} =: C_{5},
		\]
	which leads to the recurrence relation
		\[
		\Lambda_{k+1} := \left(C_{2} C_{5} + (C_1 + C_3) \Lambda_k \right)^{1/2},
		\]
	appearing in \eqref{Regularidad: Teorema Regularidad caso p entre 1 y 2}. It remains to prove that \( \Lambda_k \leq \Lambda \) for some constant \( \Lambda > 0 \) and for all \( k \geq 0 \). Define
		\[
		\Lambda := \max\left\{ \Lambda_0, \left(2 C_2 C_5 + (C_1 + C_3)^2 \right)^{1/2} \right\},
		\]
	so that the same inductive argument used in the case \( p \geq 2 \) applies in this context as well. Finally, the estimate \eqref{Regularidad: Teorema Regularidad caso p entre 1 y 2} follows by taking the limit \( k \to \infty \). 
	\end{proof}
We now turn to the complementary regime in the parameter $s$ covered by the previous theorem. More precisely, for $p>2$ we consider  \(s \in \left(0,\tfrac{1}{p'}\right),\) while for $1<p<2$ we take \(s \in \left(0,\tfrac12\right).\) Although the resulting regularity exponents are not optimal, it is worth emphasizing that the regularity obtained is monotone with respect to $s$ and coincides with the critical threshold, namely $s=\tfrac{1}{p'}$ for $p>2$ and $s=\tfrac{1}{2}$ for $1<p<2$.

	\begin{theorem}
	\label{teo-reg-subopt}
	Let \(\Omega\) be a bounded Lipschitz domain, \( \sigma \in (0,1)\), \(f \in B_{p^{\prime},1}^{0}(\Omega),\)  and \(u \in \widetilde{X}^{s,p}(\Omega)\) be a weak solution to \eqref{Regularidad:Problema Dirichlet}. If \( p \geq 2, \, s \in (0,\frac{1}{p'})\),  then \(u \in \dot{B}_{p, \infty}^{s+\frac{s}{p-1}}(\Omega)\) and
		\[
		\|u\|_{\dot{B}_{p, \infty}^{\,s + \frac{s}{p-1}}(\Omega)}^p 
		\lesssim 
		\|f\|_{B^{0}_{p',1}(\Omega)}^{p'}.
		\] 
				
In contrast, if \( 1 < p < 2, \, s \in (0,\frac{1}{2})\) then  \(u \in \dot{B}_{p, \infty}^{2s}(\Omega)\) and
		\[
		\|u\|_{\dot{B}_{p, \infty}^{\,2s }(\Omega)}^p 
		\lesssim 
		\|f\|_{B^{0}_{p',1}(\Omega)}^{p'}.
		\]
		
		Above, all hidden constants depend on \( N,s,p, \) and \( \Omega. \)
	\end{theorem}
	\begin{proof}
We follow the strategy of Theorem~\ref{teo-reg-caso-opt}, emphasizing only the modifications required in the present setting. 
Starting from inequality~\eqref{Regularidad: Cota de regularidad para F} and taking $\gamma = 0$ in Proposition~\ref{prop:reg-J1}, we obtain

\begin{equation}\label{eq-reg-F-sub-opt}
	\begin{aligned}
		\omega\left(u ; \mathcal{J}, T, \mathcal{C}_{j}, \sigma\right)
		\lesssim\;
		&\|u\|_{\dot{B}^{\sigma}_{p,\infty}(\Omega)}
		\|u\|^{p-1}_{\widetilde{X}^{s,p}(\Omega)}
		+\|\grads u\|^{p}_{L^{p}(D_{3\rho}(x_0); \mathbb{R}^N)} \\
		&\quad
		+\|f\|_{B_{p^{\prime}, 1}^{0}\left(\Omega\right)}
		\|u\|_{B_{p, \infty}^{\sigma}\left(D_{3 \rho}\left(x_{j}\right)\right)}.
	\end{aligned}
\end{equation}


\medskip
\noindent\textbf{Case $p \in (2,\infty)$.}
As in the previous theorem, from \eqref{eq:reg-functional-pge2} we know that
\[
|u|_{B_{p,\infty}^{s+\sigma/p}(D_\rho(x_j))}^p
\lesssim
\omega(u;\mathcal{J},T,\mathcal{C}_j,\sigma).
\]

Combining this estimate with \eqref{eq-reg-F-sub-opt}, summing over $j=1,\ldots,J$, and using the localization property of Besov norms stated in Lemma~\ref{lemma:loc-Besov}, we obtain the following inequality, with a constant depending on the cardinality of the covering of $\Omega^\rho$:
\begin{equation}
	\label{eq-bootstraping-subopt-p>2}
	\|u\|_{\dot{B}_{p, \infty}^{s+\sigma / p}\left(\Omega\right)}^p
	\lesssim
	\|u\|_{\dot{B}^{\sigma}_{p,\infty}(\Omega)}
	\|u\|^{p-1}_{\widetilde{X}^{s,p}(\Omega)}
	+\|\grads u\|^{p}_{L^{p}(\Omega)}
	+\|f\|_{B_{p^{\prime}, 1}^{0}\left(\Omega\right)}
	\|u\|_{\dot{B}_{p, \infty}^{\sigma}\left(\Omega\right)},
\end{equation}
for all $\sigma \in (0,1)$. 
We define the sequence $\{\sigma_k\}$ recursively by
$
	\label{eq-rec-p>2-s<1p'}
	s + \frac{\sigma_k}{p} = \sigma_{k+1},
	\ 
	\sigma_0 = 0. $
The sequence $\{\sigma_k\}$ is increasing and converges to $s p'$ as $k \to \infty$, since \(\sigma_k = s p'\bigl(1 - p^{-k}\bigr).\) Therefore, proceeding from inequality~\eqref{Regularidad: Afirmación por inducción} exactly as in the previous theorem, we conclude that

\[
\|u\|_{\dot{B}_{p, \infty}^{\,s + \frac{s}{p-1}}(\Omega)}^p 
\lesssim 
\|f\|_{B^{0}_{p',1}(\Omega)}^{p'}.
\]

\textbf{Case \(1 < p < 2\).}
Recall that in the subquadratic regime we rely on inequality~\eqref{eq:reg-functional-p<2}, namely
\[
\begin{aligned}
	|u|_{B_{p,\infty}^{s+\sigma/2}(D_\rho(x_0))}^{2} 
	\lesssim 
	\bigl(\|u\|_{\widetilde{X}^{s,p}(\Omega)} + \|T_h u\|_{\widetilde{X}^{s,p}(\Omega)}\bigr)^{2-p}
	\,\omega(u; \mathcal{J}, T, D, \sigma).
\end{aligned}
\]
Arguing as in the previous theorem, and combining this estimate with the stability bound~\eqref{Regularidad: Cota de Estabilidad} and inequality~\eqref{eq-reg-F-sub-opt}, we sum over \(j = 1, \ldots, J\) and invoke the localization property of Besov norms to obtain the following global estimate:
\begin{align*}
	|u|_{B_{p,\infty}^{\,s+\sigma/2}(\Omega)}^{2} 
	&\lesssim 
	\|f\|_{X^{-s,p'}(\Omega)}^{\frac{2-p}{p-1}} 
	\Bigg(
	C_1 \|u\|_{\dot{B}^{\sigma}_{p,\infty}(\Omega)} 
	\|u\|^{p-1}_{\widetilde{X}^{s,p}(\Omega)}  \\
	&\hspace{8em}
	+ C_2 \|u\|^{p}_{\widetilde{X}^{s,p}(\Omega)} 
	+ C_3 \|f\|_{B_{p^{\prime},1}^{0}(\Omega)} 
	\|u\|_{\dot{B}_{p,\infty}^{\sigma}(\Omega)} 
	\Bigg),
\end{align*}
for every \(\sigma \in (0,1)\). 
We now set \(\{\sigma_k\}\) by
$	s + \frac{\sigma_k}{2} = \sigma_{k+1},
 \	\sigma_0 = 0, $
namely
\(
\sigma_k = 2s \left( 1 - \left(\tfrac{1}{2}\right)^k \right).
\)
The same inductive bootstrap argument developed in~\eqref{Regularidad: Teorema Regularidad caso p entre 1 y 2} applies without modification, yielding
\[
\|u\|_{\dot{B}_{p, \infty}^{\,2s }(\Omega)}^p 
\lesssim 
\|f\|_{B^{0}_{p',1}(\Omega)}^{p'}.
\]

	\end{proof}

	We observe that, as a direct consequence of Proposition~\ref{teo-reg-caso-opt}, by setting \(p = 2\) we recover the Besov regularity result obtained  in~\cite{borthagaray2023besov}. 
	This is to be expected, since in the linear case \(p=2\) the operator under consideration reduces to the fractional Laplacian, and our regularity argument relies on the same class of localized translations.

\section{Complementary results} \label{sec:complementary}
Here, we comment on a couple of variants and extensions of our main results.

\subsection{Problems with a variable diffusivity}
In first place, the regularity estimates from Theorems \ref{teo-reg-caso-opt} and \ref{teo-reg-subopt} can be extended to problems with a variable diffusivity coefficient of the form:
		\begin{equation}
	\label{Regularidad:Problema Dirichlet difusividad variable}
			\left\{
			\begin{aligned}
				-\operatorname{div}_{s}\left( A(x)|\nabla^{s} u|^{p-2} \nabla^{s} u\right) &= f\quad\text{in } \; \Omega, \\
				u &= 0 \quad\text{ in } \; \Omega^c,
			\end{aligned}
			\right.
		\end{equation}
where the function \( A: \mathbb{R}^N \rightarrow \mathbb{R} \) is assumed to be Lipschitz continuous.
		
Indeed, the proof follows the same structure, with the main difference being that the constants in the key estimates now depend on properties of \(A\). The argument from Proposition \ref{Regularidad: Regularidad F_2 Proposition} to establish inequality \eqref{Regularidad: Regularidad de F_2 cota auxiliar savare 1} can be replicated, noting that the resulting constant will depend on the \(L^\infty\)-norm of \(A\).
		
	On the other hand, the derivation of inequality \eqref{Regularidad: Regularidad de F_2 cota auxiliar savare 2} requires a slight modification. Using the convexity of \( F(x) = |x|^p \), we have:
		\begin{align*}
			\int_{\mathbb{R}^N} \left( |T_h u(x)|^{p} - |u(x)|^{p} \right) \, dx  
			&\leq \int_{\mathbb{R}^N}
            A(x) \left( |\nabla^s (T_h u)(x)|^p - |\nabla^s u(x)|^p \right) dx \\
			&\leq \int_{\mathbb{R}^N}\left[ A(x) \phi(x) | \nabla^s v_h |^p - A(x) \phi(x)|\nabla^s v|^p \right] dx \\
			&\leq \int_{\mathbb{R}^N}
            \left( A_{-h}(x) \phi_{-h}(x) - A(x) \phi(x) \right) | \nabla^s v_h |^p \, dx.
		\end{align*}
		This last term can be bounded using the Lipschitz continuity of \(A\) and \(\phi\), yielding an estimate analogous to \eqref{Regularidad: Regularidad de F_2 cota auxiliar savare 2}:
		\begin{equation}
			\int_{\mathbb{R}^N} \left( |T_h u(x)|^p - |u(x)|^p \right) \, dx \leq C |h| \|u\|_{L^{p}(D_{3\rho}(x_0))}^p,
		\end{equation}
		where the constant \(C\) now depends on the Lipschitz constant of \(A\). 
		
		Having re-established these two key inequalities, the argument from the proof of Theorems \ref{teo-reg-caso-opt} and \ref{teo-reg-subopt} can be directly applied to obtain a fully analogous result for problem \eqref{Regularidad:Problema Dirichlet difusividad variable}. Clearly, analogous estimates can be derived in case $A$ is only H\"older continuous, but will give rise to a lower regularity pickup.
		
The result obtained for problem~\eqref{Regularidad:Problema Dirichlet difusividad variable} is new even in the case \(p=2\), since the resulting operator does not coincide with the variable-diffusivity fractional Laplacian from~\cite{borthagaray2023besov}. Instead, it gives rise to a genuinely new class of nonlocal operators as discussed in \cite{shieh2015new,garcia2025fractional}. For this class, it is natural and potentially fruitful to formulate interface problems in a manner closely analogous to the classical fractional Laplacian, opening the door to a parallel interface theory within this nonlocal framework.

\subsection{Intermediate regularity}
Next, we comment on the derivation of regularity estimates under weaker data assumptions with respect to Theorems \ref{teo-reg-caso-opt} and \ref{teo-reg-subopt}.
We rely on the following nonlinear interpolation estimate, which can be found in Tartar~\cite[Theorem~I.1]{TARTAR1972469}.

	\begin{proposition}
\label{prop-nonlinearinterplation} Let $A_0 \subset A_1, B_0 \subset B_1$ be Banach spaces, $U \subset A_1$ a nonempty open set and $T: U \rightarrow B_1$ be a a function that maps $A_0 \cap U$ into $B_0$. Let us assume that there exist constants $c_0, c_1$ such that
	
	$$
	\begin{aligned}
		& \|T f\|_{B_0} \leq c_0\|f\|_{A_0}^{\alpha_0}, \quad \forall f \in A_0 \cap U, \\
		& \|T f-T g\|_{B_1} \leq c_1\|f-g\|_{A_1}^{\alpha_1}, \quad \forall f, g \in U,
	\end{aligned}
	$$
	
	\noindent for some $\alpha_0>0, \alpha_1 \in(0,1]$. Then, if $\theta \in(0,1)$ and $q \in[1, \infty], T \operatorname{maps}\left(A_0, A_1\right)_{\theta, q}$ into $\left(B_0, B_1\right)_{\eta, r}$, where $\frac{1-\eta}{\eta}=\frac{\alpha_1}{\alpha_0} \frac{1-\theta}{\theta}$ and $r=\max \left\{1, \frac{q}{(1-\eta) \alpha_0+\eta \alpha_1}\right\}$.
	\end{proposition}
	
	Using the continuous embedding
	\[
	\widetilde{X}^{s,p}(\Omega) \hookrightarrow \dot{B}_{p,\infty}^{s}(\Omega),
	\]
	and combining Theorems~\ref{teo-reg-caso-opt} and~\ref{teo-reg-subopt} with 
    Proposition~\ref{prop-nonlinearinterplation}, we obtain the following regularity result for the solution operator via real interpolation; see \cite[Theorem 2]{savare1998regularity}, \cite[Corollary 3.2]{borthagaray2024quasi}.
	
	\begin{corollary}
		Let $\Omega$ be a bounded Lipschitz domain, let $p \in (1,\infty)$, and let $\theta \in (0,1)$.  
		Then, the solution operator $f \mapsto u$ satisfies the following mapping properties.
		
		\medskip
		
		\noindent\textbf{Case $p \geq 2$.}
		\begin{align*}
			&\text{If } s \in \Bigl[\frac{1}{p'},1\Bigr)
			\quad \text{and} \quad
			f \in W^{-s+\frac{\theta}{p'},\,p'}(\Omega),
			\quad \text{then} \quad
			u \in \widetilde{W}^{s+\frac{\theta}{p},\,p}(\Omega), \\
			&\text{If } s \in \Bigl(0,\frac{1}{p'}\Bigr)
			\quad \text{and} \quad
			f \in W^{-\theta s,\,p'}(\Omega),
			\quad \text{then} \quad
			u \in \widetilde{W}^{s+\theta\frac{s}{p-1},\,p}(\Omega).
		\end{align*}
		
		\medskip
		
		\noindent\textbf{Case $1<p<2$.}
		\begin{align*}
			&\text{If } s \in \Bigl[\frac{1}{2},1\Bigr)
			\quad \text{and} \quad
			f \in W^{-s+\frac{\theta}{2},\,p'}(\Omega),
			\quad \text{then} \quad
			u \in \widetilde{W}^{s+\frac{\theta}{2},\,p}(\Omega), \\
			&\text{If } s \in \Bigl(0,\frac{1}{2}\Bigr)
			\quad \text{and} \quad
			f \in W^{-\theta s,\,p'}(\Omega),
			\quad \text{then} \quad
			u \in \widetilde{W}^{s+\theta s,\,p}(\Omega).
		\end{align*}
	\end{corollary}
	
	We emphasize that the embedding of Lions-Calderón spaces into Besov spaces is a key ingredient in order to apply the Reiteration Theorem for the real interpolation method; see \cite[Theorem~2.3]{bellido2025bessel}.  
	
	\section*{Acknowledgments}
    J. P. Borthagaray and J. C. Rueda  have been supported in part by Agencia Nacional de Investigación e Innovación (ANII, Uruguay), grant 172393.
	L. M. Del Pezzo and J. C.  Rueda  are supported by ANII, grant 181302.

        \bibliographystyle{abbrv}  
	\bibliography{bibliografiaBO} 
	
\end{document}